\documentclass[11pt]{article}

\usepackage[utf8]{inputenc}
\usepackage[bbgreekl]{mathbbol}
\usepackage{graphicx}

\usepackage{amsmath,amssymb}
\usepackage{amsthm}

\usepackage{url}
\usepackage[numbers,square]{natbib}
\usepackage{multirow}
\usepackage{tikz}
\usepackage{tikz-cd}
\usetikzlibrary{arrows,matrix}
\usepackage{bbm}

\usepackage{comment}

\newtheorem{theorem}{Theorem}[section]
\newtheorem{corollary}[theorem]{Corollary}
\newtheorem{lemma}[theorem]{Lemma}

\theoremstyle{definition}

\newtheorem{assumption}[theorem]{Assumption}

\theoremstyle{remark}
\newtheorem{remark}[theorem]{Remark}

\numberwithin{equation}{section}

\def\bs{\boldsymbol}

\def\cp{c_{P}}
\def\cph{c_{P,h}}
\def\bcp{\bar c_{P}}
\newcommand{\poplus}{\overset{\perp}{\oplus}}
\def\ualp{\underline{\alpha}}

\newcommand{\perpW}{{\perp_W}}
\newcommand{\perph}{{\perp_h}}

\def\frB{\mathfrak B}
\def\frH{\mathfrak H}
\def\frZ{\mathfrak Z}

\def\vp{{\varphi}}
\def\ve{{\varepsilon}}

\def\be{{\bs e}}

\def\bi{{\bs i}}

\def\bk{{\bs k}}
\def\bsm{{\bs m}}
\def\bn{{\bs n}}

\def\bv{{\bs v}}

\def\bx{{\bs x}}

\def\cL{\mathcal{L}}
\def\cM{\mathcal{M}}

\def\NN{\mathbbm{N}}
\def\PP{\mathbbm{P}}
\def\QQ{\mathbbm{Q}}
\def\RR{\mathbbm{R}}

\def\VV{\mathbbm{V}}

\providecommand{\abs}[1]{\lvert#1\rvert}

\providecommand{\norm}[1]{\lVert#1\rVert}

\def\tb{\hbox{$\|\kern -.09em |$}}
\def\bigtb{\hbox{$\big\|\kern -.09em \big|$}}
\def\Bigtb{\hbox{$\Big\|\kern -.09em \Big|$}}

\providecommand{\sprod}[2]{\langle#1,#2\rangle}

\providecommand{\range}[2]{\{ #1 \dots #2 \}}

\newcommand\blue[1]{{\color{blue} #1}}

\def\bbb{\phantom{\big|} }

\DeclareMathOperator{\Div}{div}

\DeclareMathOperator{\Span}{Span}
\DeclareMathOperator{\Ima}{Im}

\DeclareMathOperator{\curl}{curl}
\DeclareMathOperator{\bcurl}{{\bf curl}}

\DeclareMathOperator{\bgrad}{{\bf grad}}

\def\matD{\mathbb{D}}

\def\matI{\mathbb{I}}

\def\matL{\mathbb{L}}
\def\matM{\mathbb{M}}

\def\matP{\mathbb{P}}

\def\matS{\mathbb{S}}

\newcommand{\arr}[1]{{\bs{\mathsf{#1}}}}

\newcommand\uvec{{\be}}

\usepackage{mathrsfs}

\newcommand\ttc{{\mathtt{c}}}
\newcommand\tte{{\mathtt{e}}}

\newcommand\ttg{{\mathtt{g}}}

\newcommand\ttn{{\mathtt{n}}}

\newcommand\ttG{{\mathtt{G}}}

\usepackage{hyperref}

\begin{document}

\title{Broken-FEEC discretizations and Hodge Laplace problems}


\author{Martin Campos Pinto and Yaman Güçlü}

\date{%
    Max-Planck-Institut für Plasmaphysik, Boltzmannstr. 2, 85748 Garching, Germany%
    \\[2ex]%
    \today
}
%




\maketitle

\begin{abstract}
  This article studies structure-preserving discretizations of Hilbert complexes with
  nonconforming (broken) spaces that rely on projection operators onto
  an underlying conforming subcomplex.
  This approach follows the conforming/nonconforming Galerkin (CONGA)
  method introduced in 
  \cite{Campos-Pinto.Sonnendrucker.2016.mcomp,Campos-Pinto.Sonnendrucker.2017a.jcm,%
  Campos-Pinto.Sonnendrucker.2017b.jcm} 
  to derive efficient structure-preserving 
  finite element schemes for the time-dependent Maxwell 
  and Maxwell-Vlasov systems by relaxing the curl-conforming constraint
  in finite element exterior calculus (FEEC) spaces.
  Here, it is extended to the discretization of full Hilbert complexes 
  with possibly nontrivial harmonic fields, and  
  the properties of the resulting CONGA Hodge Laplacian operator
  are investigated.
  
  By using block-diagonal mass matrices which may be locally inverted,
  this framework possesses a canonical sequence of dual commuting projection
  operators which are local in standard finite element applications, 
  and it naturally yields local discrete 
  coderivative operators, in contrast to conforming FEEC discretizations.
  The resulting CONGA Hodge Laplacian operator is also local,  
  and its kernel consists of the same discrete harmonic fields 
  as the underlying conforming operator,
  provided that a symmetric stabilization term is added to handle the space nonconformities.
  
  Under the assumption that the underlying conforming subcomplex admits
  a bounded cochain projection, and that the conforming projections are stable
  with moment-preserving properties, a priori convergence results are
  established for both the CONGA Hodge Laplace source and
  eigenvalue problems. Our theory is finally illustrated with a spectral element method,
  and numerical experiments are performed which corroborate our results.
  Applications to spline finite elements on multi-patch mapped domains are described in
  a related article 
  \cite{conga_psydac}, 
  for which the present work provides 
  a theoretical background.
\end{abstract}

\tableofcontents

\section{Introduction}

Over the last few decades, an important body 
of work has been devoted to the development of compatible finite element methods that
preserve the structure of de Rham complexes involved in fluid and
electromagnetic models. In addition to providing faithful approximations of
the Hodge-Helmholtz decompositions at the discrete level, such
discretizations indeed possess intrinsic stability and spectral correctness properties
\cite{Bossavit.1998.ap, Hiptmair.2002.anum, Compatible.2006.IMA, Boffi.2010.anum, buffa2010isogeometric}.
A notable step has been the unifying analysis of finite element exterior calculus (FEEC)
\cite{Arnold.Falk.Winther.2006.anum, Arnold.Falk.Winther.2010.bams}
developped in the general framework of Hilbert complexes with further applications in solid mechanics,
and where the existence of bounded cochain projections, i.e.~sequences of commuting projection operators
with uniform stability properties, has been identified as a key ingredient for discrete stability 
and structure preservation.

More recently, structure-preserving discretizations have been extended to 
nonconforming (broken) finite element spaces associated to sequences of 
conforming subspaces via stable projection operators.
The primary motivation for this was to 
improve the computational efficiency of numerical approximations to
time-dependent Maxwell \cite{Campos-Pinto.Sonnendrucker.2016.mcomp} and Maxwell-Vlasov equations
\cite{Campos-Pinto.Sonnendrucker.2017a.jcm, Campos-Pinto.Sonnendrucker.2017b.jcm},
where conforming FEEC schemes with high order elements, non-cartesian coordinates 
or non-scalar permittivities usually require a global inversion of the mass matrix,
which in turn results in the discrete coderivatives being global operators.
This difficulty is naturally resolved with broken spaces as the mass matrices 
become block diagonal.
In contractible domains where the de Rham sequence is exact, 
the resulting conforming/nonconforming Galerkin (CONGA) method has been shown to have long time stability, 
be spectrally correct and preserve key physical invariants such as the Gauss laws, 
without requiring numerical stabilization mechanisms as commonly used in discontinuous Galerkin schemes.

In this article we extend these works in several directions. 
First, we consider the discretization of full Hilbert complexes with general Hodge cohomology, and we exhibit a canonical mechanism to build stable commuting projections 
for the dual (weak) discrete complex. This shows in particular that broken FEEC discretizations provide a ready-to-use framework for nonconforming Hamiltonian 
particle approximations to Maxwell-Vlasov equations, with either
strong or weak particle-field coupling.
We then study the associated CONGA Hodge Laplacian operator with a stabilization term
for the space nonconformity. For arbitrary positive values of the stabilization parameter,
we find that this operator has the same kernel as its conforming counterpart, namely discrete harmonic fields,
and we establish several decompositions of the broken spaces that generalize the discrete 
Hodge-Helmholtz decompositions of conforming FEEC spaces. 
The associated source problem is next shown to be well-posed.
Under the assumption that the conforming projection operators are uniformly stable 
with moment-preserving properties, we establish a priori error estimates 
that allow us to recover the main stability and convergence properties 
of conforming FEEC approximations.
For stronger penalization regimes, our error estimates 
also show the spectral correctness of the CONGA Hodge Laplacian operator.
Finally we describe an application to polynomial finite elements, where this framework
naturally yields local discrete differential operators for both the primal (strong) and dual (weak) sequences, as well as local $L^2$-stable dual commuting projection operators.

We point out that for Cartesian meshes or low-order elements, there exist lumping methods based on
approximate quadrature rules which allow one to derive local approximations of the inverse mass matrices,
see e.g. \cite{Cohen_Monk_1998_nmpde,Egger_Radu_2021_sinum}, as well as
local dual differential operators \cite{Lee_Winther_2018_mcomp,Lee_2022_m2an}.
While the extension of these methods to high-order elements on unstructured or curvilinear cells is yet unclear,
the CONGA method has no such limitations:
the theory presented in this article naturally extends to curvilinear grids
(see~\cite{conga_psydac} for an application to spline finite elements on multi-patch mapped domains)
as well as unstructured grids (following the same lines as in
\cite{Campos-Pinto.Sonnendrucker.2017a.jcm,Campos-Pinto.Sonnendrucker.2017b.jcm}).

The outline is as follows. After recalling the main ingredients of conforming
FEEC discretizations of closed Hilbert complexes in Section~\ref{sec:feec},
we describe its extension on broken spaces with projection-based differential operators
in Section~\ref{sec:b_feec}, where the CONGA Hodge Laplacian operator is also presented.
The source and eigenvalue problems are then studied in Section~\ref{sec:pbms}, 
where the a priori convergence results are established. We conclude with an 
application to polynomial finite elements in Section~\ref{sec:cartfem}, and 
exhibit numerical results which confirm some of our theoretical findings.

\section{Hilbert complexes and FEEC discretizations}
\label{sec:feec}

Following \cite{Arnold.Falk.Winther.2010.bams} we consider a closed Hilbert complex
$(W, d) = (W^\ell, d^\ell)_{\ell \in \NN}$
involving unbounded, closed operators $d^{\ell} : W^\ell \to W^{\ell+1}$ with dense domains
$V^\ell$ and closed images $d^{\ell}V^\ell \subset \ker d^{\ell+1}$ implying in particular
$d^{\ell+1}d^\ell =0$.
Denoting by $\norm{\cdot}$ the Hilbert norms of the $W$ spaces, and droping the
$\ell$ indices when they are clear from the context,
the domain spaces are equipped with the graph norm
$\norm{v}^2_{V} = \norm{v}^2 + \norm{dv}^2$,
which makes the domain complex
\begin{equation} \label{seq}
V^{\ell-1} ~ \xrightarrow{\hspace*{2pt} \displaystyle d^{\ell-1} \hspace*{2pt}} ~
V^{\ell} ~ \xrightarrow{\hspace*{2pt} \displaystyle d^{\ell} \hspace*{2pt}} ~
V^{\ell+1}
\end{equation}
a bounded Hilbert complex.
By identifying each $W^\ell$ space with its dual,
we obtain a dual complex (denoted with lower indices to reflect its reverse order)
\begin{equation} \label{seqdual}
V^*_{\ell-1} ~ \xleftarrow{\hspace*{2pt} \displaystyle d^*_{\ell} \hspace*{2pt}} ~
V^*_{\ell}   ~ \xleftarrow{\hspace*{2pt} \displaystyle d^*_{\ell+1} \hspace*{2pt}} ~
V^*_{\ell+1}
\end{equation}
where the operators $d^* = d^*_{\ell+1} : W^{\ell+1} \to W^{\ell}$
are the unbounded adjoints of the $d^\ell$'s. We remind that they are characterized by the relations
\begin{equation} \label{d*}
  \sprod{d^*_{\ell+1} w}{v} = \sprod{w}{d^{\ell} v} \qquad \forall w \in V^*_{\ell+1}, ~ v \in V^{\ell}
\end{equation}
on their domains $V^*_{\ell+1} := \{ w \in W^{\ell+1} : \abs{\sprod{w}{d^\ell v}} \le C_w \norm{v}, \forall v \in V^\ell\}$,
see e.g. \cite{Brezis.2011.fa}, where $\sprod{\cdot}{\cdot}$ denotes the Hilbert product in the $W$ spaces, so that \eqref{d*} essentially amounts to an integration by parts with no boundary terms.
As in \cite{Arnold.Falk.Winther.2010.bams} we denote the ranges and kernels of the primal operators by
\begin{equation*}
\frB^\ell := \Ima(d^{\ell-1}) = dV^{\ell-1} \subset V^{\ell}
\qquad \text{ and } \qquad
\frZ^\ell := \ker(d^{\ell}) \subset  V^{\ell}
\end{equation*}
and similarly for the dual operators,
\begin{equation*}
\frB^*_\ell := \Ima(d^*_{\ell+1}) = d^*V^*_{\ell+1} \subset V^*_{\ell}
\qquad \text{ and } \qquad
\frZ^*_\ell := \ker(d^*_{\ell}) \subset  V^*_{\ell}.
\end{equation*}
Since the operators are closed and densely defined with closed ranges we have
\begin{equation*}
\frB^\ell = (\frZ^*_\ell)^{\perpW}
\qquad \text{ and } \qquad
\frB^*_\ell = (\frZ^\ell)^{\perpW}
\end{equation*}
see \cite{Brezis.2011.fa}, where $\perpW$ denotes the orthogonal complement in the proper $W$ space.

\subsection{Hodge Laplacian operator}
\label{sec:HL}

The Hodge Laplacian operator
\begin{equation} \label{L}
    L := dd^* + d^* d
\end{equation}
is a self-adjoint unbounded operator $L^\ell = d^{\ell-1} d^*_\ell + d^*_{\ell+1} d^\ell: W^\ell \to W^\ell$ with domain
\begin{equation}\label{DL}
D(L^\ell) = \{ u \in V^{\ell} \cap V^*_{\ell} : d^\ell u \in V^*_{\ell+1}, d^*_\ell u \in V^{\ell-1} \}.
\end{equation}
Its kernel and image spaces read
\begin{equation*}
  \ker L^\ell = (\frB^\ell)^\perpW \cap \frZ^\ell =: \frH^\ell
  \qquad \text{ and } \qquad
  \Ima L^\ell = (\frH^\ell)^{\perpW}
\end{equation*}
where $\frH^\ell$ is the space of harmonic fields.
If the latter is not trivial, the
source problem
\begin{equation} \label{HL_f}
L^\ell u = f
\end{equation}
is ill-posed, but it can be corrected by projecting a general source $f \in W^\ell$
and constraining the solution. The resulting problem consists of finding
$u \in (\frH^\ell)^{\perpW}$ such that $L^\ell u = f-Q_{\frH}f$,
where $Q_{\frH}$ is the $W$-orthogonal projection on $\frH^\ell$. It may be recast in a mixed form:
\begin{equation} \label{HL_mixed}
  \begin{aligned}
    &\text{ Find }\quad  (\sigma, u, p) \in X := V^{\ell-1} \times V^\ell \times \frH^\ell, \quad \text{ such that }
    \\
  &\left\{ \begin{aligned}
    \sprod{\sigma}{\tau} - \sprod{d\tau}{u} &= 0 \qquad &&\forall \tau \in V^{\ell-1}
    \\
    \sprod{d\sigma}{v} + \sprod{du}{dv} + \sprod{v}{p} &= \sprod{f}{v} \qquad &&\forall v \in V^{\ell}
    \\
    \sprod{u}{q} &= 0 \qquad &&\forall q \in \frH^{\ell}.
  \end{aligned}\right.
  \end{aligned}
\end{equation}
An equivalent formulation is to find $(\sigma, u, p) \in X$ such that
\begin{equation}\label{HL_b_form}
  b(\sigma, u, p; \tau,v,q) = \sprod{f}{v} \qquad \forall (\tau,v,q) \in X 
\end{equation}
with the bilinear form
\begin{equation} \label{b}
b(\sigma, u, p; \tau,v,q) := \sprod{\sigma}{\tau} - \sprod{d\tau}{u} + \sprod{d\sigma}{v} + \sprod{du}{dv}
    + \sprod{v}{p}
    -  \sprod{u}{q}.
\end{equation}
The well-posedness of this problem essentially relies on the closed complex property which
leads to a generalized Poincaré inequality of the form 
\begin{equation}\label{pcr}
  \norm{v}_V \le \cp \norm{dv}, \qquad v \in V^\ell \cap (\frZ^\ell)^{\perpW}
\end{equation}
using Banach's bounded inverse theorem.
It is indeed shown in \cite[Th.~3.2]{Arnold.Falk.Winther.2010.bams} that for all
$(\sigma, u, p) \in X$, there exists $(\tau, v, q) \in X$ for which
\begin{equation}\label{bstab}
b(\sigma, u, p; \tau,v,q) \ge \gamma (\norm{\sigma}_V + \norm{u}_V+\norm{p})(\norm{\tau}_V + \norm{v}_V+\norm{q})
\end{equation}
holds with a constant $\gamma > 0$ depending only on the Poincaré constant $\cp$.
Noting that $b(\sigma, u, p; \tau,v,q) = b(\tau, -v, q; \sigma,-u,p)$, this
leads to the inf-sup condition
\begin{equation}\label{infsup}
  \inf_{y \in X}\sup_{x\in X} \frac{b(x,y)}{\norm{x}_X\norm{y}_X} \ge \gamma > 0
\end{equation}
where 
  $X = V^{\ell-1} \times V^\ell \times \frH^\ell$
is equipped with 
$  \norm{(\tau,v,q)}_X = \norm{\tau}_V + \norm{v}_V+\norm{q}$.
Classically, the inf-sup condition \eqref{infsup} implies that the operator
$B : X \to X'$ defined by $\sprod{Bx}{y}_{X'\times X} = b(x,y)$ is
surjective~\cite{Boffi.Brezzi.Fortin.2013.scm}, and it is clearly
injective by \eqref{bstab}. Hence \eqref{HL_b_form} admits a unique solution, which satisfies
\begin{equation*}
\norm{\sigma}_V + \norm{u}_V+\norm{p} \le \gamma^{-1}\norm{f}.
\end{equation*}

\subsection{Conforming FEEC discretization}
\label{sec:hc}

The usual discretization is provided by a finite dimensional subcomplex of the form
\begin{equation} \label{seq_hc}
V^{\ell-1,c}_h ~ \xrightarrow{\hspace*{2pt} \displaystyle d^{\ell-1,c}_h \hspace*{2pt}} ~
V^{\ell,c}_h ~ \xrightarrow{\hspace*{2pt} \displaystyle d^{\ell,c}_h \hspace*{2pt}} ~
V^{\ell+1,c}_h
\end{equation}
where the discrete differential operators are the restrictions of the continuous ones,
\begin{equation} \label{dlc}
  d^{\ell,c}_h = d^{\ell} : V^{\ell,c}_h \to V^{\ell+1,c}_h~.
\end{equation}
Here the $c$ superscript indicates that these discrete spaces are {\em conforming} in the sense that $V^{\ell,c}_h \subset V^\ell$.
This notation somehow deviates from the usual one ($V^\ell_h$) 
which we will use for the broken spaces in Section~\ref{sec:b_feec}, since they are the focus of this article.
A key result \cite{Arnold.Falk.Winther.2006.anum,Arnold.Falk.Winther.2010.bams} is that the stability of the
conforming discretization relies on the existence of a {\em bounded cochain projection} $\pi_h$,
i.e. projection operators $\pi^\ell_h: V^\ell \to V^{\ell,c}_h$ that satisfy the commuting diagram property
\begin{equation}\label{pih}
  d^\ell \pi^\ell_h = \pi^{\ell+1}_h d^\ell \qquad \text{ on } ~ V^\ell
\end{equation}
such that $\norm{\pi_h v}_V \le \norm{\pi_h}_V \norm{v}_V$ for all $v \in V^\ell$, 
with an operator norm $\norm{\pi_h}_V$ bounded independent of $h$.
Throughout the article the notation $\norm{\cdot}_V$ will be used for both the $V$ norm and the operator norm in $V$.

\begin{remark}[discretization parameter $h$] \label{rem:h}
  Here and below, the subscript $h$ loosely represents a discretization parameter 
  that can be varied to improve the resolution of the discrete spaces. 
  Typically this parameter corresponds to a mesh size, but
  in most places (and unless specified otherwise)  
  it may represent arbitrary discretization parameters.
  What matters in the analysis is that several properties will hold with constants independent of it.
  Classically these constants will be denoted with the generic letter $C$,
  whose value may change at each occurrence.
\end{remark}

The dual discrete complex involves the same discrete spaces: it reads
\begin{equation} \label{seqdual_hc}
V^{\ell-1,c}_h ~ \xleftarrow{\hspace*{2pt} \displaystyle d^{*,c}_{\ell,h} \hspace*{2pt}} ~
V^{\ell,c}_h   ~ \xleftarrow{\hspace*{2pt} \displaystyle d^{*,c}_{\ell+1,h} \hspace*{2pt}} ~
V^{\ell+1,c}_h
\end{equation}
where
$d^{*,c}_{\ell+1,h} : V^{\ell+1,c}_h \to V^{\ell,c}_h$
is the adjoint of $d^{\ell,c}_h$, i.e. 
\begin{equation} \label{wdlc}
\sprod{d^{*,c}_{\ell+1,h} q}{v} = \sprod{q}{d^{\ell} v}, \qquad \forall v \in V^{\ell,c}_h.
\end{equation}
In similarity to the continuous case (and omitting again the $\ell$ indices when they are clear
from the context), we denote the discrete kernels and ranges by
\begin{equation} \label{BZh_c}
\frB^{\ell,c}_h := \Ima d^{\ell-1,c}_{h} = d^{c}_{h} V^{\ell-1,c}_{h}
\qquad \text{ and } \qquad
\frZ^{\ell,c}_h := \ker d^{\ell,c}_{h}
\end{equation}
and similarly for the discrete adjoint operators,
\begin{equation} \label{BZh*_c}
\frB^{*,c}_{\ell,h} := \Ima d^{*,c}_{\ell+1,h} = d^{*,c}_{h} V^{\ell+1,c}_{h}
\qquad \text{ and } \qquad
\frZ^{*,c}_{\ell,h} := \ker d^{*,c}_{\ell,h}
\end{equation}
where we notice that these four spaces are all subspaces of $V^{\ell,c}_h$.
As the discrete operators are all closed and bounded with closed range, we also have
\begin{equation} \label{B=Z*p_c}
\frB^{\ell,c}_h = V^{\ell,c}_h \cap (\frZ^{*,c}_{\ell,h})^\perpW =: \frZ^{*,c \perp}_{\ell,h}
\end{equation}
and similarly
\begin{equation} \label{B*=Zp_c}
\frB^{*,c}_{\ell,h} = V^{\ell,c}_h \cap (\frZ^{\ell,c}_h)^\perpW =: \frZ^{\ell,c \perp}_h~.
\end{equation}
It follows from the basic sequence property $d^{\ell,c}_h d^{\ell+1,c}_h = 0$ that
$\frB^{\ell,c}_h \subset \frZ^{\ell,c}_h$, in particular $\frB^{\ell,c}_h$ and $\frB^{*,c}_{\ell,h} = \frZ^{\ell,c \perp}_h$
are orthogonal.
Denoting the complement space as
\begin{equation} \label{H_c}
\frH^{\ell,c}_h := (\frB^{\ell,c }_h)^\perpW \cap \frZ^{\ell,c}_h \subset V^{\ell,c}_h,
\end{equation}
yields a discrete Hodge-Helmholtz decomposition for the conforming space,
\begin{equation} \label{HHdec_c}
  V^{\ell,c}_h = \frB^{\ell,c}_h \overset{\perp}{\oplus} \frH^{\ell,c}_h \overset{\perp}{\oplus} \frB^{*,c}_{\ell,h}.
\end{equation}

\subsection{Discrete Hodge Laplacian operator}

The conforming discrete Hodge Laplacian operator is
\begin{equation} \label{Lh_c}
    L^{c}_h := d^{c}_h d^{*,c}_{h} + d^{*,c}_{h} d^{c}_h
\end{equation}
specifically,
$L^{c}_h = L^{\ell,c}_h = d^{\ell-1,c}_h d^{*,c}_{\ell,h} + d^{*,c}_{\ell+1,h} d^{\ell,c}_h : V^{\ell,c}_h \to V^{\ell,c}_h$.
Its kernel consists of the {\em discrete harmonic fields}, namely the $v \in V^\ell_h$ 
for which both $d^{*,c}_{\ell,h} = 0$ and $d^{\ell,c}_h v = 0$.
From \eqref{BZh_c}--\eqref{HHdec_c} one infers that
\begin{equation} \label{kerL_c}
  \ker L^{\ell,c}_h = \frZ^{*,c}_{\ell,h} \cap \frZ^{\ell,c}_h  =  \frH^{\ell,c}_h.
\end{equation}
Using the fact that $d^c_h = d$ on the discrete conforming spaces,
the corresponding source problem in mixed form reads:
\begin{equation} \label{HL_mixed_hc}
  \begin{aligned}
    &\text{ Find }\quad  (\sigma^c_h, u^c_h, p^c_h) \in X^c_h := V^{\ell-1,c}_h \times V^{\ell,c}_h \times \frH^{\ell,c}_h \quad \text{ such that }
    \\
  &\left\{ \begin{aligned}
    \sprod{\sigma^c_h}{\tau} - \sprod{d\tau}{u^c_h} &= 0 \qquad &&\forall \tau \in V^{\ell-1,c}_h
    \\
    \sprod{d\sigma^c_h}{v} + \sprod{du^c_h}{dv} + \sprod{v}{p^c_h} &= \sprod{f}{v} \qquad &&\forall v \in V^{\ell,c}_h
    \\
    \sprod{u^c_h}{q} &= 0 \qquad &&\forall q \in \frH^{\ell,c}_h~.
    \end{aligned}\right.
  \end{aligned}
\end{equation}
The stability of Problem~\eqref{HL_mixed_hc} then relies on a discrete Poincaré inequality 
\begin{equation} \label{pcr_hc} 
  \norm{v}_V \le \cph \norm{d v}, \qquad v \in \frZ^{\ell,c \perp}_h
\end{equation}
see \eqref{B*=Zp_c}, which itself follows from the existence of a bounded cochain projection. 
Indeed the following results holds, see \cite[Th.~3.6 and 3.8]{Arnold.Falk.Winther.2010.bams},
which leads to the well-posedness of Problem~\eqref{HL_mixed_hc} in similarity to the continuous case.
\begin{theorem} \label{th:HL_c}
  If the conforming discrete complex $(V^c_h, d)$ admits a
  $V$-bounded cochain projection \eqref{pih},
  then the discrete Poincaré inequality \eqref{pcr_hc} holds with a constant 
  $\cph = \cp \norm{\pi_h}_V$ where $\cp$ is from \eqref{pcr}.
  Moreover, for any
  $(\sigma, u, p) \in X^c_h = V^{\ell-1,c}_h \times V^{\ell,c}_h \times \frH^{\ell,c}_h$,
  there exists $(\tau,v,q) \in X^c_h$ such that
  \begin{equation*}
    b(\sigma, u, p; \tau,v,q) \ge \gamma (\norm{\sigma}_V + \norm{u}_V+\norm{p})(\norm{\tau}_V + \norm{v}_V+\norm{q})
  \end{equation*}
  holds for some $\gamma > 0$ depending only on the discrete Poincaré constant $\cph$.
  In particular, Problem~\eqref{HL_mixed_hc} is well-posed.
\end{theorem}
We finally remind that Problem~\eqref{HL_mixed_hc} is equivalent to finding
$u^c_h \in V^{\ell,c}_h \cap (\frH^{\ell,c}_h)^\perpW$ such that
$L^{\ell,c}_h u^c_h = Q_{V^c_h} f - Q_{\frH^c_h} f$,
with $\sigma^c_h = d^{*,c}_{\ell,h} u^c_h$ and $p^c_h = Q_{\frH^c_h} f$.
(Throughout the article $Q_U$ denotes the $W$-orthogonal projection onto a 
closed subspace $U$.)
Indeed, each component of the solution in the discrete decomposition
\begin{equation} \label{uc_dec}
  u^c_h = u^c_\frB + u^c_\frH + u^c_{\frB^*}
  ~ \in ~ \frB^{\ell,c}_h \poplus \frH^{\ell,c}_h \poplus \frB^{*,c}_{\ell,h}
\end{equation}
may be characterized by taking test functions in the suitable subspaces: we have
\begin{equation} \label{HL_mixed_hc_B}
  \left\{\begin{aligned}
   \sprod{\sigma^c_h}{\tau} - \sprod{d\tau}{u^c_\frB} &= 0 \qquad &&\forall \, \tau \in V^{\ell-1,c}_h
   \\
   \sprod{d \sigma^c_h}{v} &= \sprod{f}{v}  \qquad &&\forall \, v \in \frB^{\ell,c}_h~,
 \end{aligned} \right.
\end{equation}
the harmonic component is $u^c_\frH = 0$, and
\begin{equation} \label{HL_mixed_hc_B*}
  \sprod{d u^c_{\frB^*}}{d v} = \sprod{f}{v} \qquad \forall \, v \in \frB^{*,c}_{\ell,h}~.
\end{equation}
Taking $v \in \frH^{\ell,c}_{h}$ finally yields
$p^c_h = Q_{\frH^c_h} f$.

\section{Broken FEEC discretization}
\label{sec:b_feec}

We now study a discretization of the Hilbert complex \eqref{seq} where the conformity 
requirement is relaxed. 
Specifically, we assume in this section that we are given a conforming discretization \eqref{seq_hc}
and we consider at each level $\ell$ a discrete space $V^{\ell}_h \subset W^\ell$ that contains 
the conforming one,
\begin{equation} \label{Vlh}
  V^{\ell,c}_h \subset V^{\ell}_h
\end{equation}
but is not necessarily a subspace of $V^\ell$. The typical situation that we have in mind
is the one where the conforming spaces have some finite element structure of the form
$V^{\ell,c}_h = \big( \VV^{\ell}(\Omega_1) \times \dots \times \VV^{\ell}(\Omega_K)\big) \cap V^\ell$
on a partition of the domain $\Omega$ into subdomains $\Omega_k$, $k = 1, \dots K$, 
and where the conformity $V^{\ell,c}_h \subset V^\ell$
amounts to continuity constraints on the subdomain interfaces,
see e.g. \cite{Boffi.Brezzi.Fortin.2013.scm,conga_psydac} and Section~\ref{sec:cartfem}.
Working with the broken spaces
$V^{\ell}_h = \VV^{\ell}(\Omega_1) \times \dots \times \VV^{\ell}(\Omega_K)$
allows one to lift these constraints, yielding more locality for the discrete operators
and more flexibility in the numerical modelling.

Since our analysis is based on a stable conforming discretization, we make this assumption explicit.
\begin{assumption}\label{as:bcp} %
  The conforming sequence $V^c_h$ 
  admits a uniformly $V$-bounded cochain projection \eqref{pih}.
  According to Theorem~\ref{th:HL_c} this implies that the discrete Poincaré
  inequality \eqref{pcr_hc} holds with a constant
  \begin{equation} \label{bcp}
    0 < \cph \le \bcp
  \end{equation}
  where $\bcp$ is independent of the discretization parameter $h$.
\end{assumption}

\subsection{Projection-based differential operators} 
\label{sec:bsp}

A discrete Hilbert complex involving the broken spaces can be
obtained by considering projection operators onto the conforming subspaces,
\begin{equation} \label{Pc}
  P^{\ell}_h : V^{\ell}_h  \to V^{\ell,c}_h \subset V^{\ell}_h
\end{equation}
and by defining discrete differential operators on the broken spaces as
\begin{equation}\label{dh}
d^\ell_h := d^\ell P^{\ell}_h~. 
\end{equation}
These operators map $V^{\ell}_h$ to $V^{\ell+1,c}_h \subset V^{\ell+1}_h$, and by the projection
property they satisfy $d^\ell_h d^{\ell-1}_h = d^\ell d^{\ell-1}P^{\ell-1}_h = 0$, hence
we indeed obtain a discrete Hilbert complex,
\begin{equation} \label{seq_h}
  V^{\ell-1}_h \xrightarrow{ \mbox{$~ d^{\ell-1}_h ~$}}
    V^{\ell}_h \xrightarrow{ \mbox{$~ d^{\ell}_h ~$}}
      V^{\ell+1}_h.
\end{equation}
This construction may be summarized by the following diagram where the horizontal
sequences are Hilbert complexes and the vertical arrows denote projection operators.
\begin{equation*}
  \begin{tikzpicture}[ampersand replacement=\&, baseline] 
  \matrix (m) [matrix of math nodes,row sep=3em,column sep=4em,minimum width=2em] {
          ~~ V^{\ell-1} ~ \bbb
              \& ~~ V^{\ell} ~ \bbb
                  \& ~~ V^{\ell+1} ~ \bbb
    \\
    ~~ V^{\ell-1,c}_h ~ \bbb
        \& ~~ V^{\ell,c}_h ~ \bbb
            \& ~~ V^{\ell+1,c}_h ~ \bbb
    \\
    ~~ V^{\ell-1}_h ~ \bbb
        \& ~~ V^{\ell}_h ~ \bbb
            \& ~~ V^{\ell+1}_h ~ \bbb
    \\
  };
  \path[-stealth]
  (m-1-1) edge node [above] {$d^{\ell-1}$} (m-1-2)
          edge node [right] {$\pi^{\ell-1}_h$} (m-2-1)
  (m-1-2) edge node [above] {$d^{\ell}$} (m-1-3)
          edge node [right] {$\pi^{\ell}_h$} (m-2-2)
  (m-1-3) edge node [right] {$\pi^{\ell+1}_h$} (m-2-3)
  (m-2-1) edge node [above] {$d^{\ell-1}$} (m-2-2)
  (m-2-2) edge node [above] {$d^{\ell}$} (m-2-3)
  (m-3-1) edge node [above] {$d^{\ell-1}P^{\ell-1}_h$} (m-3-2)
          edge node [right] {$P^{\ell-1}_h$} (m-2-1)
  (m-3-2) edge node [above] {$d^\ell P^{\ell}_h$} (m-3-3)
          edge node [right] {$P^{\ell}_h$} (m-2-2)
  (m-3-3) edge node [right] {$P^{\ell+1}_h$} (m-2-3)
  ;
  \end{tikzpicture}
\end{equation*}
We note that this diagram commutes, since the strong differential operators
map onto the conforming spaces.
For the subsequent analysis we equip the broken spaces with Hilbert norms
\begin{equation} \label{normVh}
  \norm{v}_{V_h}^2 = \norm{v}^2 + \norm{dP_hv}^2, \qquad v \in V^\ell_h,
\end{equation}
and assume that the operators $P_h$ are bounded uniformly in $W$, namely that
\begin{equation} \label{Phstab}
  \norm{P_h v} \le C \norm{v} \quad \forall v \in V^\ell_h
\end{equation}
holds with a constant independent of $h$, see Remark~\ref{rem:h}.
Again, for conciseness we sometimes drop the level indices $\ell$ when they are clear from the context.

As in the conforming case, a dual discrete sequence is built on the same spaces
\begin{equation} \label{seqdual_h}
V^{\ell-1}_h ~ \xleftarrow{\hspace*{2pt} \displaystyle d^{*}_{\ell,h} \hspace*{2pt}} ~
V^{\ell}_h   ~ \xleftarrow{\hspace*{2pt} \displaystyle d^{*}_{\ell+1,h} \hspace*{2pt}} ~
V^{\ell+1}_h
\end{equation}
by introducing the discrete adjoint operators
$d^{*}_{\ell+1,h} := (d^\ell_h)^*: V^{\ell+1}_h \to V^{\ell}_h$
defined as
\begin{equation} \label{dualdh}
\sprod{d^{*}_{\ell+1,h} q}{v} = \sprod{q}{d^{\ell}P^\ell_h v}, \qquad \forall v \in V^{\ell}_h.
\end{equation}
The CONGA (broken FEEC) Hodge Laplacian operator is then defined as
\begin{equation} \label{Lh}
    L_{h} := d_h d^*_{h} + d^*_{h} d_h
\end{equation}
and a stabilized version is
\begin{equation} \label{Lha}
  L_{h,\alpha} := L_{h} + \alpha (I-P_h^*)(I-P_h) 
\end{equation}
namely, $L^\ell_{h,\alpha} = d^{\ell-1}_h d^*_{\ell,h} + \alpha (I-(P^\ell_h)^*)(I-P^\ell_h) + d^*_{\ell+1,h} d^\ell_h: V^\ell_h \to V^\ell_h$.
Here the stabilization term involves the adjoint $P_h^* := (P^\ell_h)^* : V^\ell_h \to V^\ell_h$ of the
discrete conforming projection, and a parameter $\alpha \ge 0$. 
In the regimes where $\alpha \to \infty$ as the discretization parameter $h$ is refined,
this term may be seen as a penalization of the nonconformities.
However, an arbitrary positive stabilization is sufficient to recover the conforming 
harmonic fields \eqref{H_c} as the kernel of the broken operator.
\begin{theorem} \label{th:HL_ker}
  The CONGA Hodge Laplacian operator is a symmetric positive semi-definite operator in $V^\ell_h$. 
  For a stabilization parameter $\alpha > 0$, its kernel coincides with that of 
  the conforming operator \eqref{kerL_c}, i.e.
  \begin{equation} \label{kerLh}
  \ker L^\ell_{h,\alpha} = \frH^{\ell,c}_h
  \end{equation}
  and its image is 
  \begin{equation} \label{ImaLh}
  \Ima L^\ell_{h,\alpha} = \big(\frH^{\ell,c}_h\big)^\perph
  \end{equation}
  where the $\perph$ exponent on a discrete space denotes the $W$-orthogonal complement in the natural 
  broken space $V^\ell_h$.  
\end{theorem}
The proof of this result relies on some decompositions of the nonconforming space $V^\ell_h$,
which we will present in Section~\ref{sec:dec_h}. Before doing so, we describe 
projections operators that commute with the dual differentials.

\subsection{Commuting diagrams for strong and weak broken FEEC complexes}

Before turning to the analysis of the Hodge Laplacian operator \eqref{Lh},
we formalize and extend an observation previously made in
\cite{Campos-Pinto.Sonnendrucker.2016.mcomp}, where it was shown that
the adjoint of the conforming projection composed with the (local) $L^2$ projection
onto the broken $H(\curl)$ space commutes with the weak CONGA curl operator.
In the full broken FEEC setting considered here this principle is generalized
to the construction of a canonical sequence of stable projections
that commute with the dual differential operators. 
These dual projections are defined as 
\begin{equation} \label{tpih}
  \tilde \pi^{\ell}_{h} := (P^{\ell}_h)^* Q_{V^{\ell}_h} : W^\ell \to V^\ell_h
\end{equation}
where we remind that $Q_{V^{\ell}_h}$ is the $W$-orthogonal projection onto $V^{\ell}_h$. Namely they are characterized by the relations
\begin{equation} \label{tpih_char}
  \sprod{\tilde \pi^{\ell}_{h}w}{v} = \sprod{w}{P^{\ell}_h v}, \quad \forall w \in W^\ell, ~ v \in V^\ell_h.
\end{equation}
Together with the stable commuting projection operators $\pi^\ell_h$ available 
for the conforming spaces $V^{\ell,c}_h \subset V^\ell_h$, 
this leads to a commuting diagram for both the primal
(strong) and dual (weak) complexes.
\begin{theorem} \label{th:tpih}
  The operators $\tilde \pi^{\ell}_{h}$ are uniformly $W$-stable projections onto the spaces
  \begin{equation} \label{ima_tpih} 
    \Ima \tilde \pi^\ell_h = (P^{\ell}_h)^* V^{\ell}_h 
    = \{ v \in V^\ell_h : \sprod{v}{(I-P^\ell_h)w} = 0, ~ \forall w \in V^\ell_h\}~.
  \end{equation}
  Moreover they commute with the dual differential operators:
  \begin{equation} \label{dual_comm}
    d^*_{\ell,h} \tilde \pi^\ell_h = \tilde \pi^{\ell-1}_h d^*_\ell \qquad \text{ on } ~ V^*_\ell~.
  \end{equation}
  In particular, under Assumption~\ref{as:bcp} %
  we find that 
  both the primal (top) and dual (bottom) diagram below commute. 
\begin{equation*}
  \begin{tikzpicture}[ampersand replacement=\&, baseline] 
  \matrix (m) [matrix of math nodes,row sep=3em,column sep=5em,minimum width=2em] {
        ~~ V^{\ell-1} ~ \bbb
            \& ~~ V^{\ell} ~ \bbb
                \& ~~ V^{\ell+1} ~ \bbb
    \\
    ~~ V^{\ell-1}_h ~ \bbb
        \& ~~~ V^{\ell\phantom{+1}}_h \bbb   
            \& ~~ V^{\ell+1}_h ~ \bbb
    \\
    ~~ V^*_{\ell-1} ~ \bbb
        \& ~~ V^*_{\ell} ~ \bbb
            \& ~~ V^*_{\ell+1} ~ \bbb
    \\
  };
  \path[-stealth]
  (m-1-1) edge node [above] {$d^{\ell-1}$} (m-1-2)
          edge node [right] {$\pi^{\ell-1}_h$} (m-2-1)
  (m-1-2) edge node [above] {$d^{\ell}$} (m-1-3)
          edge node [right] {$\pi^{\ell}_h$} (m-2-2)
  (m-1-3) edge node [right] {$\pi^{\ell+1}_h$} (m-2-3)
  (m-2-1.10) edge node [above] {$d^{\ell-1}_h$} (m-2-2.170)
  (m-2-2.10) edge node [above] {$d^{\ell}_h$} (m-2-3.170)
  (m-2-2.190) edge node [below] {$d^*_{\ell,h}$} (m-2-1.350)
  (m-2-3.190) edge node [below] {$d^*_{\ell+1,h}$} (m-2-2.350)
  (m-3-2) edge node [above] {$d^*_{\ell}$} (m-3-1)
  (m-3-3) edge node [above] {$d^*_{\ell+1}$} (m-3-2)
  (m-3-1) edge node [right] {$\tilde \pi^{\ell-1}_{h}$} (m-2-1)
  (m-3-2) edge node [right] {$\tilde \pi^{\ell}_{h}$} (m-2-2)
  (m-3-3) edge node [right] {$\tilde \pi^{\ell+1}_{h}$} (m-2-3)
  ;
  \end{tikzpicture}
\end{equation*}
\end{theorem}

\begin{remark}
  The projections $\tilde \pi^\ell_h$ are the broken-FEEC analogue of the $W$-orthogonal projection operators which commute with the dual differential operators 
  in the conforming FEEC model.
  A key point is that here the broken nature of the spaces $V^\ell_h$
  naturally leads to dual projection operators that are local
  when applied to standard finite elements spaces
  (see for instance Theorem~\ref{th:loc_bfeec} below),
  in contrast to what happens in the conforming case. 
\end{remark}

\begin{proof}[Proof.~]
  The $W$ stability is easily derived from that of $P^\ell_h$, indeed \eqref{Phstab} allows us to write
  $\sprod{\tilde \pi^{\ell}_{h}w}{v} \le \norm{w}\norm{P^{\ell}_h v} \le C \norm{w}\norm{v}$,
  hence $\norm{\tilde \pi^{\ell}_{h}w} \le C \norm{w}$ with the same constant as in \eqref{Phstab}.
  The projection property is also straightforward: given that $P^\ell_h$ is itself a projection,
  we see that $(\tilde \pi^{\ell}_{h})^2 w \in V^\ell_h$ is characterized by
  $$
  \sprod{(\tilde \pi^{\ell}_{h})^2w}{v} 
  = \sprod{\tilde \pi^{\ell}_{h} w}{P^{\ell}_h v}
  = \sprod{w}{(P^{\ell}_h)^2 v} 
  = \sprod{w}{P^{\ell}_h v} \quad \forall ~ v \in V^\ell_h
  $$
  hence $(\tilde \pi^{\ell}_{h})^2 w = \tilde \pi^{\ell}_{h} w$,
  and \eqref{ima_tpih} follows from the fact that $\Ima \tilde \pi^\ell_h = \Ima (P^\ell_h)^* = (\ker P^\ell_h)^{\perph}$
  (the symbol $\perph$ was introduced in Theorem~\ref{th:HL_ker})
  and $\ker P^\ell_h = \Ima (I-P^\ell_h)$. Finally the commuting property is
  a consequence of the weak definition of the dual differential operators: 
  indeed for $v \in V^*_\ell$ and $\tau \in V^{\ell-1}_h$ it holds 
  \begin{multline*}
    \sprod{d^*_{\ell,h} \tilde \pi^\ell_h v}{\tau} 
    = \sprod{\tilde \pi^\ell_h v}{d^{\ell-1}_h\tau}
    = \sprod{v}{P^{\ell}_h d^{\ell-1}_h\tau}
    = \sprod{v}{d^{\ell-1}_h\tau}
    \\
    = \sprod{v}{d^{\ell-1} P^{\ell-1}_h\tau}
    = \sprod{d^*_{\ell} v}{P^{\ell-1}_h\tau}
    = \sprod{\tilde \pi^{\ell-1}_h d^*_{\ell} v}{\tau}  
  \end{multline*}
  where the third equality uses the fact that $d^{\ell-1}_h$ maps into the conforming space
  $V^{\ell,c}_h$ where $P^{\ell}_h$ is the identity, and the fifth one uses the adjoint
  property \eqref{d*} and the fact that $P^{\ell-1}_h\tau$ is in $V^{\ell-1,c}_h$, hence in 
  $V^{\ell-1}$.
\end{proof}

An important by-product of our analysis is that broken-FEEC Maxwell solvers 
may be used to derive structure-preserving particle schemes, following the GEMPIC 
approach \cite{kraus2016gempic,CPKS_variational}. Indeed the latter applies to general 
commuting de Rham diagrams, with no assumption of conformity.

\begin{corollary}
  By applying the variational discretization method from \cite{CPKS_variational}
  to either the primal broken-FEEC sequence \eqref{seq_h} or the dual one \eqref{seqdual_h},  
  and its associated primal or dual commuting projection operators,
  one obtains a Hamiltonian particle discretization of the Vlasov-Maxwell system
  with broken spaces for the field solver.
\end{corollary}

\subsection{Broken Hodge-Helmholtz decompositions}
\label{sec:dec_h}

As with the conforming operators, we define
\begin{equation} \label{BZh*}
\left\{ \begin{aligned}
    &\frB^{\ell}_h := \Ima d^{\ell-1}_h
    \\
    &\frB^{*}_{\ell,h} := \Ima d^{*}_{\ell+1,h} 
\end{aligned} \right.
\quad \text{ and } \qquad
\left\{ \begin{aligned}
  &\frZ^{\ell}_h := \ker d^\ell_h
  \\
  &\frZ^{*}_{\ell,h} := \ker d^{*}_{\ell,h}~.
\end{aligned} \right.
\end{equation}
These spaces may be related with the conforming ones in several ways.
First, using \eqref{dh} and the fact that $P^\ell_h$ is a projection onto the conforming space $V^{\ell,c}_h$, 
we observe that 
\begin{equation} \label{B=Bc}
  \frB^{\ell}_h = d^{\ell-1} P^{\ell-1}_h V^{\ell-1}_h = d^{\ell-1} V^{\ell-1,c}_h = \frB^{\ell,c}_h
\end{equation}
see \eqref{BZh_c}. Next by using the analysis in \cite{Campos-Pinto.2016.cras} and \eqref{H_c} we obtain
\begin{equation} \label{Zdec}
  \frZ^{\ell}_h = \frZ^{\ell,c}_h \oplus (I-P_h^\ell)V^\ell_h = \frB^{\ell,c}_h \oplus  \frH^{\ell,c}_h \oplus (I-P_h)V^\ell_h
\end{equation}
which also yields
\begin{equation} \label{Zdec2}
  \frZ^{\ell}_h \cap V^{\ell,c}_h = \frZ^{\ell,c}_h~.
\end{equation}
Using the $\perph$ exponent to denote $W$-orthogonal complements in the natural $V^\ell_h$ space, 
as introduced in Theorem~\ref{th:HL_ker},
we then write analogs to \eqref{B*=Zp_c} and \eqref{B=Z*p_c}, namely
\begin{equation} \label{BB*=Zp}
\frB^{\ell}_{h} = (\frZ^*_{\ell,h})^\perph =: \frZ^{* \perp}_{\ell,h}
\qquad \text{ and } \qquad
\frB^*_{\ell,h} = (\frZ^{\ell}_h)^\perph =: \frZ^{\ell,\perp}_h
\end{equation}
and similarly we observe that
\begin{equation} \label{VcpP*}
(I-P_h^*)V^{\ell}_h = \Ima (I-P_h^\ell)^* = (\ker (I-P_h^\ell))^\perph = (V^{\ell,c}_h)^\perph~.
\end{equation}
These relations allow us to study the kernel of the stabilized CONGA Hodge Laplacian operator. 
\begin{proof}[Proof of Theorem~\ref{th:HL_ker}.~]
  The first statement is obvious, since
  $L^\ell_{h,\alpha}$ (and $L^\ell_h$) is a sum of symmetric positive semi-definite operators.
  To show \eqref{kerLh}, we test $L^\ell_{h,\alpha} u = 0$ against $u$: this
  yields 
  $$
  0 = \sprod{L^\ell_{h,\alpha}u}{u} = \norm{d^*_{\ell,h}u}^2 + \alpha\norm{(I-P_h)u}^2 + \norm{d^\ell_h u}^2.
  $$
  For $\alpha > 0$ we thus have
  \begin{equation*}
    \ker L^\ell_{h,\alpha} = 
    \ker d^*_{\ell,h} \cap \ker (I-P_h) \cap \ker d^\ell_h
    = \frZ^*_{\ell,h} \cap V^{\ell,c}_h \cap \frZ^\ell_h
    = \frZ^*_{\ell,h} \cap \frZ^{\ell,c}_h
  %
  \end{equation*}
  where the last equality is \eqref{Zdec2}.
  Using next \eqref{BB*=Zp} and \eqref{B=Bc} gives 
  $\frZ^*_{\ell,h} = (\frB^{\ell,c}_h)^\perph$, hence
  \begin{equation*}
  \ker L^\ell_{h,\alpha}
    = (\frB^{\ell,c}_h)^\perph \cap \frZ^{\ell,c}_h
    = (\frB^{\ell,c}_h)^\perpW \cap \frZ^{\ell,c}_h
    = \frH^{\ell,c}_h
  \end{equation*}
  according to the definition of the discrete harmonic forms \eqref{H_c}.
  Finally \eqref{ImaLh} follows from the usual property $\Ima A = (\ker A)^\perph$
  for a symmetric operator $A:V_h \to V_h$.
\end{proof}

We conclude this section by establishing some generalized Hodge-Helmholtz decompositions 
for the broken spaces.
\begin{lemma} \label{lem:HHdec}
  The broken space $V^\ell_h$ admits one orthogonal decomposition: 
  \begin{equation} \label{HHdec_1}
    V^\ell_h 
    = \frB^{\ell,c}_h \poplus  \frH^{\ell,c}_h \poplus \frB^{*,c}_{\ell,h} \poplus (I-P_h^*)V^{\ell}_h
  \end{equation}
  and several non-orthogonal ones:
  \begin{align}
    \label{HHdec_2}
    &V^\ell_h 
      = \frB^{\ell,c}_h \oplus  \frH^{\ell,c}_h \oplus \frB^{*,c}_{\ell,h} \oplus (I-P_h)V^{\ell}_h
  \\
  \label{HHdec_3}
    &V^\ell_h
    = \frB^{\ell}_h \oplus  \frH^{\ell,c}_h \oplus \frB^{*}_{\ell,h} \oplus (I-P_h)V^\ell_h 
  \\ \noalign{\smallskip} \label{HHdec_4}
    &V^\ell_h 
    = \frB^{\ell}_h \oplus \frH^{\ell,c}_h \oplus \frB^{*}_{\ell,h} \oplus (I-P_h^*)V^{\ell}_h
\end{align}
where we remind that $\frB^{\ell}_h = \frB^{\ell,c}_h$, see \eqref{B=Bc}.
\end{lemma}
\begin{proof}[Proof.~]
  The first decomposition follows from 
  writing $V^\ell_h = V^{\ell,c}_h \poplus (V^{\ell,c}_h)^\perph$ and using 
  \eqref{VcpP*} together with the conforming decomposition \eqref{HHdec_c}.
  Similarly we derive \eqref{HHdec_2} from $V^\ell_h = V^{\ell,c}_h \oplus (I-P_h)V^\ell_h$,
  and \eqref{HHdec_3} is easily obtained from \eqref{Zdec} and the second relation in \eqref{BB*=Zp}.
To show the last decomposition \eqref{HHdec_4} we start from \eqref{HHdec_1} and write an arbitrary $u \in V^\ell_h$ as
\begin{equation*}
u = d \rho + r + d^{*,c}_h \phi + (I-P_h^*)w
\end{equation*}
with $\rho \in V^{\ell-1,c}_h$, $r \in \frH^{\ell,c}_h$, $\phi \in V^{\ell+1,c}_h$ and $w \in V^\ell_h$.
Since for all $v \in V^\ell_h$ we have
\begin{equation*}
\sprod{d^{*}_h \phi}{v} = \sprod{\phi}{dP_h v} = \sprod{d^{*,c}_h\phi}{P_h v}
  = \sprod{P^*_h d^{*,c}_h\phi}{v}
\end{equation*}
we infer that $d^{*}_h \phi = P_h^* d^{*,c}_h \phi$, hence
$
u = d \rho + r + d^{*}_h \phi + (I-P_h^*)(w + d^{*,c}_h \phi)
$
which establishes the sum $V^\ell_h = \frB^{\ell,c}_h + \frH^{\ell,c}_h + \frB^{*}_{\ell,h} + (I-P_h^*)V^{\ell}_h$.
To show that this sum is direct, given \eqref{HHdec_1} and \eqref{HHdec_3}
it suffices to show that the last two spaces are disjoint, or equivalently that
$\frZ^{\ell,\perp}_h \cap (I-P_h^*)V^{\ell}_h = \{0\}$, see \eqref{BB*=Zp}.
This is verified by observing that any $v = (I-P_h^*)v \in \frZ^{\ell,\perp}_h$ satisfies
$
\norm{v}^2 = \sprod{(I-P_h^*)v}{v} =\sprod{v}{(I-P_h)v} = 0
$
where we have used the fact that $(I-P_h)V^\ell_h = \ker P_h^\ell \subset \ker d_h^\ell = \frZ^{\ell}_h$.
\end{proof}

\section{Analysis of broken-FEEC Hodge Laplace problems}
\label{sec:pbms}

We now turn to the analysis of broken-FEEC 
approximations to Hodge Laplace source and eigenvalue problems.
In this article we shall focus on the properties of the stabilized 
Hodge Laplacian operator $L^\ell_{h,\alpha_h}$ 
with a positive parameter $\alpha_h$ bounded away from zero.
Throughout this section we make the following assumption, in addition to Assumption~\ref{as:bcp}.
\begin{assumption}\label{as:ualp}
  The stabilization parameter $\alpha_h$ satisfies
  \begin{equation} \label{ualp}
  \alpha_h \ge \ualp > 0
  \end{equation}
  for some constant $\ualp$ independent of the discretization parameter $h$.
\end{assumption}

\subsection{The Hodge Laplace source problem}

In our broken-FEEC framework we approximate 
the source problem \eqref{HL_mixed} using a product space
\begin{equation} \label{Xh}
  X_h := V^{\ell-1}_h \times V^\ell_h \times \frH^{\ell,c}_h
\end{equation}
equipped with a Hilbert norm derived from that of the broken spaces \eqref{normVh},
\begin{equation} \label{normXh}
  \norm{(\tau, v, q)}_{X_h}^2 := \norm{\tau}_{V_h}^2 + \norm{v}_{V_h}^2 + \norm{q}^2.
\end{equation}
We then consider the following mixed problem:
Given $f \in W^\ell$, find $(\sigma_h,u_h,p_h) \in X_h$, such that
\begin{equation} \label{HL_mixed_h}
   \left\{\begin{aligned}
    \sprod{\sigma_h}{\tau} - \sprod{dP_h \tau}{u_h} &= 0 
    \\
    \sprod{dP_h \sigma_h}{v} + \sprod{dP_h u_h}{dP_h v} + \alpha_h\sprod{(I-P_h)u_h}{(I-P_h)v} + \sprod{P_hv}{p_h} &= \sprod{f}{P_h v} 
    \\
    \sprod{P_hu_h}{q} &= 0 
  \end{aligned}\right.
\end{equation}
for all $(\tau, v, q) \in X_h$. Here the filtering of the source by the adjoint 
conforming projection $P_h^*$ corresponds to using the dual commuting projection 
\eqref{tpih} in the source approximation.
This is motivated by the structure-preserving properties of the first CONGA method 
developped for Maxwell equations in \cite{Campos-Pinto.Sonnendrucker.2016.mcomp}, 
and is convenient to avoid introducing source approximation errors in the a priori 
error analysis below. Another option is to replace
\begin{equation} \label{unfilter_f}
  \sprod{f}{P_h v} ~ \rightarrow ~ \sprod{f}{v}
\end{equation}
in \eqref{HL_mixed_h}, which corresponds to a simple $L^2$ projection for the source
and will be useful for the study of the eigenvalue problem.
This is better seen by rewriting the mixed problem in operator form.

\begin{lemma} \label{lem:HL_ops_h}
  Problem~\eqref{HL_mixed_h} amounts to finding $u_h \in 
  (P_h^*\frH^{\ell,c}_h)^\perph = V^\ell_h \cap \big(P_h^*\frH^{\ell,c}_h\big)^\perpW$, such that 
  \begin{equation} \label{HL_ops_h}
    L^\ell_{h,\alpha_h} u_h = f_h 
  \end{equation}
  with $f_h = \tilde \pi^\ell_h (f - Q_{\frH^{c}_h} f)$, 
  or $f_h = Q_{V_h} f - P_h^* Q_{\frH^c_h} f$ 
  in the case of an unfiltered source \eqref{unfilter_f}.
  Here we remind that $\tilde \pi^\ell_h$ is the dual commuting projection \eqref{tpih} 
  and $Q_{V_h}$, resp $Q_{\frH^c_h}$, is the $W$-orthogonal projection 
  onto $V^{\ell}_h$, resp $\frH^{\ell,c}_h$.
  The remaining parts of the solution are then given by
  \begin{equation} \label{ps}
    p_h = Q_{\frH^c_h} f \quad \text{ and } \quad \sigma_h = d^*_{\ell,h} u_h.    
  \end{equation}
\end{lemma}

\begin{proof}[Proof.~]
  We begin by observing that the first equation in \eqref{HL_mixed_h} 
  amounts to $\sigma_h = d^*_{\ell,h} u_h$ thanks to \eqref{dualdh}, 
  and that the last one amounts to the constraint that $u_h$ is in the orthogonal complement of $P_h^*\frH^{\ell,c}_h$.
  Testing the second equation with $v \in \frH^{\ell,c}_h$ then yields $p_h = Q_{\frH^c_h} f$
  (both in the filtered and unfiltered cases).
  Finally the equivalence between \eqref{HL_ops_h} and the second equation from 
  \eqref{HL_mixed_h}, using \eqref{ps}, is easily derived from the definition of the CONGA
  Hodge Laplacian operator \eqref{Lh}--\eqref{Lha}.
\end{proof}

\begin{remark} \label{rem:unfilter_lhs}
  In addition to the simple (unfiltered) $L^2$ projection of the source described in \eqref{unfilter_f},
  one may consider an unfiltered projection of the harmonic terms in the left-hand side, i.e., replace
  \begin{equation} \label{unfilter_lhs}
    \sprod{P_hv}{p_h} ~ \rightarrow ~ \sprod{v}{p_h} \quad \text{ and } \quad \sprod{P_hu_h}{q}  ~ \rightarrow ~ \sprod{u_h}{q}
  \end{equation}
  in \eqref{HL_mixed_h}. This option corresponds to finding $u_h \in (\frH^{\ell,c}_h)^\perph$ such that 
  $L^\ell_{h,\alpha_h} u_h = f_h$ with $f_h = Q_{V_h}(f - Q_{\frH^c_h} f)$, and $p_h$, $\sigma_h$ given again by \eqref{ps}.
  This problem admits a unique solution and it leads to a stable approximation, but under a slightly stronger condition:
  see Remarks~\ref{rem:sol_unfilter_lhs} and \ref{rem:stab_unfilter_lhs}.  
\end{remark}

Before turning to the actual stability analysis, we observe that the existence and uniqueness of a 
solution is easily infered from Lemma~\ref{lem:HL_ops_h}.
\begin{lemma}
  Both Problem~\eqref{HL_mixed_h} and its ``unfiltered'' version \eqref{unfilter_f}
  admit a unique solution.
\end{lemma}
\begin{proof}[Proof.~] We will show that there exists a unique solution to \eqref{HL_ops_h}
  satisfying the proper orthogonality constraint and the result will follow from Lemma~\ref{lem:HL_ops_h}.
  Using the orthogonal projections, \eqref{tpih} and the fact that $P_h$ is the identity on 
  $\frH^{\ell,c}_h \subset V^{\ell,c}_h$, we first verify that both the sources 
  $f_h = \tilde \pi^\ell_h (f - Q_{\frH^{c}_h} f)$ and $f_h = Q_{V_h} f - P_h^* Q_{\frH^c_h} f$ 
  belong to $(\frH^{\ell,c}_h)^\perph$.
  Since $\alpha_h > 0$ by assumption~\ref{as:ualp}, Theorem~\ref{th:HL_ker} applies
  and this shows that there exists $v_h \in V^\ell_h$ such that 
  $L^\ell_{h,\alpha_h} v_h = f_h$. Let then $u_h = v_h - Q_{\frH^c_h}P_h v_h$.
  This function still satisfies $L^\ell_{h,\alpha_h} u_h = f_h$ since $u_h-v_h \in \frH^{\ell,c}_h$
  (again by Theorem~\ref{th:HL_ker}) and it also satisfies
  the orthogonal constraint since for all $q \in \frH^{\ell,c}_h$ it holds
  $
  \sprod{u_h}{P_h^*q} 
  = \sprod{v_h - Q_{\frH^c_h}P_h v_h}{P_h^*q} 
  = \sprod{P_h v_h - Q_{\frH^c_h}P_h v_h}{q} = 0
  $
  where we have used again that $P_h$ is the identity
  on $\frH^{\ell,c}_h$.
  To show the uniqueness, further assume that $L^\ell_{h,\alpha_h} u_h = 0$. 
  This would imply $u_h \in \frH^{\ell,c}_h$, and using the orthogonal constraint 
  $u_h \in \big(P_h^*\frH^{\ell,c}_h\big)^\perpW$ we would have
  $0 = \sprod{u_h}{P_h^*u_h} = \sprod{P_h u_h}{u_h} = \norm{u_h}^2$.
  This shows that there exists a unique solution to \eqref{HL_ops_h} that satisfies the orthogonal 
  constraint and ends the proof.
\end{proof}

\begin{remark} \label{rem:sol_unfilter_lhs}
  One can verify that the ``fully unfiltered'' problem described in Remark~\ref{rem:unfilter_lhs} 
  also admits a unique solution, by a straightforward adaptation of the above arguments (the source $f_h = Q_{V_h}(f - Q_{\frH^c_h} f)$ belongs to $(\frH^{\ell,c}_h)^\perph$, 
  so that $L^\ell_{h,\alpha_h} v_h = f_h$ holds for some $v_h \in V^\ell_h$, and $u_h := v_h - Q_{\frH^c_h}v_h$ 
  is a solution which is orthogonal to the kernel of $L^\ell_{h,\alpha_h}$).
\end{remark}

To study the well-posedness of \eqref{HL_mixed_h} we recast it in the form 
\begin{equation}\label{HLh_b_form}
  b_h(\sigma_h, u_h, p_h; \tau,v,q) = \sprod{f}{P_h v} \qquad  \forall (\tau,v,q) \in X_h
\end{equation}
with a new bilinear form on $X_h$ defined as
\begin{multline} \label{bh}
  b_h(\sigma_h, u_h, p_h; \tau,v,q) := \sprod{\sigma_h}{\tau} - \sprod{dP_h \tau}{u_h} + \sprod{dP_h \sigma_h}{v} + \sprod{dP_h u_h}{dP_h v} \\
      + \alpha_h \sprod{(I-P_h)u_h}{(I-P_h)v}+ \sprod{P_hv}{p_h} - \sprod{P_hu_h}{q}.
\end{multline}
We note that the unfiltered version \eqref{unfilter_f} corresponds to 
\begin{equation}\label{HLh_b_unfiltered}
  b_h(\sigma_h, u_h, p_h; \tau,v,q) = \sprod{f}{v}, \qquad  \forall (\tau,v,q) \in X_h.
\end{equation}
Using \eqref{Phstab} we verify that $b_h$ is continuous on $X_h$, namely
\begin{equation} \label{bcont}
  b_h(x,y) \le C (1+\alpha_h)\norm{x}_{X_h}\norm{y}_{X_h} \qquad \forall x, y \in X_h~.
\end{equation}
Thus, the continuity constant may depend on $h$ through $\alpha_h$. 
However it is easily verified that this dependency disappears if one considers 
functions in the conforming subspace $V^{\ell,c}_h$.
We then have the following result.
\begin{lemma} \label{lem:bstab_h}
  For all $(\sigma, u, p) \in X_h$,
  there exists  $(\tau,v,q) \in X_h$
  such that
  \begin{equation} \label{bstab_h}
    b_h(\sigma, u, p; \tau,v,q) \ge \gamma (\norm{\sigma}_{V_h} + \norm{u}_{V_h}+\norm{p})(\norm{\tau}_{V_h} + \norm{v}_{V_h}+\norm{q})
  \end{equation}
  holds for some $\gamma > 0$ which only depends on $\ualp$ and $\bcp$ from \eqref{ualp} and \eqref{bcp}, moreover $(I-P_h)v = (I-P_h)u$.
\end{lemma}
\begin{remark} \label{rem:stab_unfilter_lhs}
  A similar stability result holds for the bilinear form corresponding to an unfiltered projection of the harmonic terms as described in 
  Remark~\ref{rem:unfilter_lhs}, under the stronger condition $\frac 14 < \ualp$.
\end{remark}
\begin{proof}[Proof.~] We extend the proof of \cite[Th.~3.2]{Arnold.Falk.Winther.2010.bams} to the
case of broken spaces.
According to 
\eqref{HHdec_c} or \eqref{HHdec_2},
any $u \in V^\ell_h$ decomposes into
\begin{equation*}
  u = P_h u + (I-P_h)u
  = u_\frB + u_\frH + u_{\frB^*}
  + (I-P_h)u
  ~ \in ~ \big(\frB^{\ell,c}_h \poplus \frH^{\ell,c}_h \poplus \frB^{*,c}_{\ell,h} \big) \oplus (I-P_h)V^\ell_h
\end{equation*}
and with $\rho \in \frZ^{\ell-1,c\perp}_h$ such that $d\rho = u_\frB$,
the discrete Poincaré inequality \eqref{pcr_hc} yields 
\begin{equation}\label{pcr_rho_up}
  \norm{\rho}_{V_h} = \norm{\rho}_V \le \bcp \norm{u_\frB}, \qquad
  \norm{u_{\frB^*}}_{V_h} = \norm{u_{\frB^*}}_V \le \bcp \norm{dP_h u}  
\end{equation}
where we have used \eqref{bcp}, $P_h\rho = \rho$ and $dP_h u_{\frB^*} = du_{\frB^*} = dP_h u$.
We then set
\begin{equation*}
\tau = \sigma - \frac {1}{\bcp^2}\rho \in V^{\ell-1}_h,
\qquad
v = u + dP_h \sigma + p \in V^{\ell}_h,
\qquad
q = p - u_\frH  \in \frH^{\ell,c}_h
\end{equation*}
and we first infer from \eqref{pcr_rho_up} and $(\norm{u_\frB}^2 + \norm{u_\frH}^2)^{\frac 12} \le \norm{P_h u} \le C \norm{u}$,
see \eqref{Phstab}, that
\begin{equation} \label{tvq}
  \norm{\tau}_{V_h} + \norm{v}_{V_h}+\norm{q} 
  \le C (\norm{\sigma}_{V_h} + \norm{u}_{V_h} + \norm{p})
\end{equation}
with a constant independent of $h$. Recalling that the $V_h$ norm \eqref{normVh} involves $dP_h$, we then compute
\begin{equation*}
\begin{aligned}
  b_h(\sigma, u, p; \tau,v,q)
  &= \sprod{\sigma}{\sigma - \frac {1}{\bcp^2}\rho} + \frac {1}{\bcp^2}\sprod{d\rho}{u}
    + \sprod{dP_h \sigma}{dP_h \sigma} + \sprod{dP_h u}{dP_h u}
  \\
      & \mspace{50mu} + \alpha_h \sprod{(I-P_h)u}{(I-P_h)u}+ \sprod{dP_h \sigma + p}{p} + \sprod{P_hu}{u_\frH }
      \\
  &= \norm{\sigma}_{V_h}^2 - \frac{1}{\bcp^2}\sprod{\sigma}{\rho}
  + \frac{1}{\bcp^2} \big(\norm{u_\frB}^2 + \sprod{u_\frB}{(I-P_h)u}\big) + \norm{dP_h u}^2
  \\
  & \mspace{50mu} + \alpha_h \norm{(I-P_h)u}^2 + \norm{p}^2 + \norm{u_\frH}^2
\end{aligned}
\end{equation*}
where we have used in several places the orthogonality of the conforming Hodge-Helmholtz decomposition
and the fact that $d$ and $dP_h$ vanish on $\frB^c_h$ and $\frH^c_h$.
In the last sum, the products' amplitude may be bounded from above
\begin{equation*}
\frac{1}{\bcp^2} \abs{\sprod{\sigma}{\rho}} \le \frac{1}{\bcp^2} \norm{\sigma}\norm{\rho}
  \le \frac{1}{\bcp}  \norm{\sigma}\norm{u_\frB} \le \frac{1}{2} \norm{\sigma}^2
    + \frac{1}{2\bcp^2} \norm{u_\frB}^2~,
\end{equation*}
\begin{equation*}
\text{and} \qquad  \frac{1}{\bcp^2}\abs{\sprod{u_\frB}{(I-P_h)u}} \le \frac{\beta\alpha_h}{2\bcp^2} \norm{(I-P_h)u}^2 + \frac{1}{2\beta\alpha_h\bcp^2}
  \norm{u_\frB}^2
\end{equation*}
for an arbitrary $\beta > 0$. This allows us to bound the sum from below
\begin{equation*}
  \begin{aligned}
  b_h(\sigma, u, p; \tau,v,q) &\ge
  \frac 12 \norm{\sigma}_{V_h}^2 + \frac{1}{\bcp^2} \Big(1 - \frac{1}{2\beta\alpha_h} - \frac{1}{2} \Big)\norm{u_\frB}^2
  \\
  &\qquad + \norm{dP_h u}^2 + \alpha_h \Big(1 - \frac{\beta}{2\bcp^2}\Big) \norm{(I-P_h)u}^2 + \norm{p}^2
  + \norm{u_\frH}^2.
\end{aligned}
\end{equation*}
Up to using a larger constant $\bcp \leftarrow \max(\bcp, \ualp^{-1/2})$, we can assume
$\alpha^{-1}_h \le \ualp^{-1} \le \bcp^2$ so that taking $\beta = \frac 32 \bcp^2$ yields
\begin{equation*}
  b_h(\sigma, u, p; \tau,v,q)
  \ge C \big(\norm{\sigma}_{V_h}^2 + \norm{u_\frB}^2
    + \norm{dP_h u}^2 + \norm{(I-P_h)u}^2 + \norm{p}^2
  + \norm{u_\frH}^2 \big)
\end{equation*}
with $C = C(\ualp,\bcp)$.
(For the bilinear form described in Remark~\ref{rem:unfilter_lhs},
the same reasonning yields a term $\sprod{(I-P_h)u}{u_\frH}$ which may be bounded
from below by $-\frac{1}{2}(\mu\norm{(I-P_h)u}^2 + \mu^{-1}\norm{u_\frH}^2)$:
the latter can be absorbed in the above bound under the condition that $\frac{\mu}{2} < \ualp$
and $\frac{1}{2\mu} < 1$, hence the result stated in Remark~\ref{rem:stab_unfilter_lhs}.)
Since
$
\norm{u}^2 
\le C(\norm{u_\frH}^2  + \norm{u_\frB}^2 + \norm{dP_h u}^2 + \norm{(I-P_h)u}^2)
$
according to the decomposition of $u$ and \eqref{pcr_rho_up},
we find $b_h(\sigma, u, p; \tau,v,q) \ge C (\norm{\sigma}_{V_h}^2  + \norm{u}_{V_h}^2  + \norm{p}^2)$
and the desired estimate follows from \eqref{tvq}.
Finally we observe that the identity $(I-P_h)v = (I-P_h)u$ is clear in this construction.
\end{proof}

Reasoning as in Section~\ref{sec:HL}, one infers from Lemma~\ref{lem:bstab_h}
that the CONGA Hodge Laplace source problem is well-posed.
Specifically, the following result holds.

 \begin{theorem} \label{th:wp}
  Problem~\eqref{HL_mixed_h} admits a unique solution
  $(\sigma_h,u_h,p_h) \in X_h$ which satisfies
  \begin{equation} \label{ustab_h}
    \norm{\sigma_h}_{V_h} + \norm{u_h}_{V_h}+\norm{p_h} \le C \norm{f}
  \end{equation}
  with a constant which only depends on $\bcp$ and $\ualp$ from \eqref{bcp} and \eqref{ualp}. 
  The same result holds for the unfiltered variant \eqref{unfilter_f}.
\end{theorem}
According to \eqref{HHdec_2} and \eqref{HHdec_c} we can decompose the solution to Problem~\eqref{HL_mixed_h} as
\begin{equation} \label{uh_dec}
  u_h = u_\frB + u_\frH + u_{\frB^*} + (I-P_h)u_h
  ~ \in ~ \big(\frB^{\ell,c}_h \poplus \frH^{\ell,c}_h \poplus \frB^{*,c}_{\ell,h} \big) \oplus (I-P_h)V^\ell_h
\end{equation}
and observe that some components may be characterized using different types of test functions,
as with the conforming solution $u^c_h = u^c_\frB + u^c_\frH + u^c_{\frB^*}$ in \eqref{uc_dec}--\eqref{HL_mixed_hc_B*}.
One has for instance
\begin{equation} \label{HL_mixed_h_B*}
  \sprod{d u_{\frB^*}}{d v} = \sprod{f}{v} \quad \forall \, v \in \frB^{*,c}_{\ell,h}~,
\quad \text{ hence } \quad
u_{\frB^*} = u^c_{\frB^*}
\end{equation}
where we have used \eqref{HL_mixed_hc_B*},
and with $v \in \frB^{\ell,c}_h$ we obtain
\begin{equation} \label{HL_mixed_h_B}
  \left\{\begin{aligned}
   \sprod{\sigma_h}{\tau} - \sprod{dP_h \tau}{u_h} &= 0 \quad &&\forall \, \tau \in V^{\ell-1}_h
   \\
   \sprod{dP_h \sigma_h}{v} &= \sprod{f}{v}  \quad &&\forall \, v \in \frB^{\ell,c}_h~.
 \end{aligned} \right.
\end{equation}
Meanwhile, taking $v$ and $q$ in the harmonic subspace $\frH^{\ell,c}_h$ gives
\begin{equation} \label{p_uH}
  p_h = Q_{\frH^{c}_h} f
  = p^c_h 
  \quad \text{ and } \quad
  u_\frH = 0 = u^c_\frH
\end{equation}
and with $v = (I-P_h)v \in (I-P_h)V^{\ell}_h$ we further find
\begin{equation} \label{HL_mixed_h_jumps}
  \sprod{dP_h \sigma_h}{(I-P_h)v} + \alpha_h \sprod{(I-P_h)u_h}{(I-P_h)v} = 0,
  \qquad \forall \, v \in V^{\ell}_h~.
\end{equation}
From this last equality one may infer an a priori bound on the nonconforming part $(I-P_h)u_h$.
Indeed, \eqref{HL_mixed_h_B} shows that $dP_h \sigma_h = Q_{\frB^{c}_h}f$. 
Setting $v = u_h$ in \eqref{HL_mixed_h_jumps} gives then
\begin{equation} \label{jbound_fil}
  \norm{(I-P_h)u_h} \le \alpha_h^{-1} \norm{dP_h \sigma_h} \le \alpha_h^{-1} \norm{f}.
\end{equation}

\begin{remark} \label{rem:unfiltered}
  The solution to the unfiltered problem \eqref{HLh_b_unfiltered} may be decomposed as above, 
  with the only difference that \eqref{HL_mixed_h_jumps} holds 
  with a right-hand side $\sprod{f}{(I-P_h)v}$. 
  In turn, we obtain a bound similar to \eqref{jbound_fil} for the solution jumps, i.e.
    $\norm{(I-P_h)u_h} \le \alpha_h^{-1} \big(\norm{(I-Q_{\frB^{c}_h})f}\big) \le \alpha_h^{-1} \norm{f}$.
\end{remark}

\subsection{A priori error analysis}
\label{analysis}

To establish error bounds 
we now assume that the conforming projection operator $P^\ell_h$ can be extended
to the full space $W^\ell$ into some 
$\bar P_h = \bar P^\ell_h : W^\ell \to V^{\ell,c}_h$
that is uniformly bounded in $W$ and $V$, i.e. 
\begin{equation} \label{bPhstab}
  \norm{\bar P_h v} \le C \norm{v}
  \qquad \text{ and } \qquad
  \norm{\bar P_h v}_V \le C \norm{v}_V
\end{equation}
hold on $W^\ell$ and $V^\ell$ with constants independent of $h$. As a bounded operator on $W$, $\bar P_h$
has an adjoint $\bar P^*_h = (\bar P^\ell_h)^* : W^\ell \to W^\ell$
which is also a bounded projection, and its image
\begin{equation} \label{MP*}
  M^\ell_h := \bar P^*_h W^\ell
\end{equation}
corresponds to the moments preserved by $\bar P_h$, as we have
\begin{equation} \label{char_MP*}
\sprod{\bar P_h v}{w} = \sprod{v}{\bar P^*_h w} = \sprod{v}{w},
\qquad \forall v \in W^\ell, ~~ w \in M^\ell_h~.
\end{equation}
As uniformly bounded projections, these operators satisfy
\begin{equation} \label{optP}
  \norm{(I-\bar P_h) v} \le C \inf_{w \in V^{\ell,c}_h} \norm{v-w},
  \qquad 
  \norm{(I-\bar P_h) v}_V \le C \inf_{w \in V^{\ell,c}_h} \norm{v-w}_V
\end{equation}
(these estimates follow by writing 
$(I-\bar P_h)v = (I-\bar P_h)(v-w)$ for all $w \in V^{l,c}_h$, and using
\eqref{bPhstab}) 
and similarly 
\begin{equation} \label{optP*}
  \norm{(I-\bar P_h^*) v} \le C \inf_{w \in M^{\ell}_h} \norm{v-w}
\end{equation}
with constants independent of $h$. Note that no $V$-error estimate holds
for $\bar P_h^*$, as this operator may not map onto $V^\ell$.
We also introduce an extended Hilbert space
that contains both the exact and discrete solutions,
\begin{equation} \label{X(h)}
X(h) := V^{\ell-1}(h) \times V^{\ell}(h) \times W^\ell
\quad \text{ with } \quad
  \norm{(\tau, v, q)}_{X(h)}^2 := 
  \norm{\tau}_{V(h)}^2 + \norm{v}_{V(h)}^2 + \norm{q}^2
\end{equation}
where we have set $V(h) = V + V_h$ with $\norm{v}_{V(h)}^2 = \norm{v}^2 + \norm{d\bar P_h v}^2$.
We see that both $X$ and $X_h$ are indeed subspaces of $X(h)$, see \eqref{HL_mixed}, \eqref{Xh}
(their sum being a proper subspace, as $\frH^{\ell}$ and $\frH^{\ell,c}_h$ are proper subspaces of $W^\ell$)
and that the norm of $V(h)$, resp. $X(h)$, coincides with that of $V_h$, resp. $X_h$,
on the latter space, since $\bar P_h$ coincides with $P_h$ on $V_h$. From \eqref{bPhstab} one also infers
\begin{equation} \label{X(h)boundX}
\norm{(\tau, v, q)}_{X(h)} \le C (\norm{\tau}_V + \norm{v}_V + \norm{q}) \quad \text{ for }
(\tau, v, q) \in V^{\ell-1} \times V^{\ell} \times W^\ell.
\end{equation}
Accordingly, we extend the broken bilinear form $b_h$, see \eqref{bh}, into
\begin{multline} \label{bbh}
  \bar b_h(\sigma, u, p; \tau,v,q) := \sprod{\sigma}{\tau} - \sprod{d\bar P_h \tau}{u} + \sprod{d\bar P_h \sigma}{v}
    + \sprod{d\bar P_h u}{d\bar P_h v} \\
      + \alpha_h \sprod{(I-\bar P_h)u}{(I-\bar P_h)v}+ \sprod{\bar P_h v}{p} - \sprod{\bar P_h u}{q},
\end{multline}
defined on $X(h)$. This bilinear form is again continuous thanks to \eqref{bPhstab},
\begin{equation} \label{bbcont}
  \bar b_h(x,y) \le C (1+\alpha_h)\norm{x}_{X(h)}\norm{y}_{X(h)} \qquad \forall x, y \in X(h)
\end{equation}
however no stability holds on $X(h)$, due to the fact that the extented space for harmonic fields
is the full space $W^\ell$.
We further note that $\bar b_h$ coincides with $b_h$ on $X_h$ since again $\bar P_h$ coincides with $P_h$ on
$V_h$, but it does not coincide with $b$ on $X$, see \eqref{b}.
In our nonconforming framework, we then extend \cite[Th.~3.9]{Arnold.Falk.Winther.2010.bams}
as follows.
\begin{theorem} \label{th:est}
  Let $x = (\sigma, u, p)$ and $x_h = (\sigma_h, u_h, p_h)$ be the solutions to the
  continuous and discrete problems \eqref{HL_b_form} and \eqref{HLh_b_form} respectively.
  The estimate
  \begin{equation} \label{errestim}
    \norm{x - x_h}_{X_h} \le C \Big( \eta(x) + \min\big(\alpha_h\eta(x),
        \alpha_h^{-1}\norm{f}\big) \Big)
  \end{equation}
  holds with
  \begin{equation} \label{eta}
  \begin{aligned}
    \eta(x) &= \inf_{\rho \in V^{\ell-1,c}_h} \norm{\sigma - \rho}_V
          + \inf_{w \in V^{\ell,c}_h} \norm{u - w}_V
          + \inf_{w \in V^{\ell,c}_h} \norm{p - w}_V
          + \norm{Q_{\frH^c_h}u}
         \\
         & \mspace{300mu}
         + \inf_{\rho \in M^{\ell-1}_h} \norm{\sigma - \rho}
         + \inf_{w \in M^{\ell}_h} \norm{d\sigma - w}
  \end{aligned}
  \end{equation}
  and a constant independent of $h$.
\end{theorem}

\begin{remark}
  This estimate suggests two stabilization/penalization strategies. In a ``weak'' stabilization regime where
  $\alpha_h$ is uniformly bounded with $h$, the error is bounded by the error term $\eta(x)$
  which, up to the approximation errors from the moment spaces $M^{\ell-1}_h$ and $M^{\ell}_h$,
  is similar to the one involved in \cite[Th.~3.9]{Arnold.Falk.Winther.2010.bams} and leads to
  high order convergence rates for smooth solutions.
  Another option is to choose a ``strong'' penalization regime with $\alpha_h \to \infty$ 
  as $h$ is refined (see Remark~\ref{rem:h}),
  in which case high order convergence rates are also possible.
  In view of Theorem~\ref{th:refest} below, and as our numerical result suggest,
  this latter strategy seems to give better results for the eigenvalue problem.
\end{remark}

\begin{remark} \label{rem:errestim_unfiltered}
  For the unfiltered source problem \eqref{HLh_b_unfiltered}, 
  it also holds that 
  \begin{equation} \label{errestim_unfiltered}
    \norm{x - x_h}_{X_h} \le C \big( \eta(x) + \alpha_h^{-1}\norm{f}\big)
  \end{equation}
  which is a useful estimate for a strong penalization regime.
  A converging estimate for the weak stabilization regime can also be established, with
  an additional source approximation error in the upper bound.
\end{remark}

\begin{proof}[Proof of Theorem~\ref{th:est} and Remark~\ref{rem:errestim_unfiltered}.~] 
Given an arbitrary test tuple $(\tau, v, q) \in X_h$, we recall that
the discrete solution satisfies
\begin{equation} \label{baserr_bhxh}
b_h(x_h; \tau, v, q) = \sprod{f}{P_h v},
\end{equation}
wheras for the exact solution we may write
\begin{equation} \label{baserr_bx}
  b(x; P_h\tau, P_h v, q) = \sprod{f}{P_h v} - \sprod{u}{q} = \sprod{f}{P_h v} -\sprod{Q_{\frH^c_h}u}{q}
\end{equation}
the latter term being a priori non zero since $\frH^{\ell,c}_h \not\subset \frH^\ell$ in general.
To handle the discrepancy between the discrete and continuous bilinear forms
we next use the extended bilinear form \eqref{bbh} and compute that
\begin{equation} \label{baserr_bhx}
  \bar b_h(x; \tau, v, q)
    = b(x; P_h \tau, P_h v, q)
      + r(x; \tau, v, q)
\end{equation}
with a remainder term
\begin{equation} \label{rem}
  \begin{aligned}
    r(x; \tau, v, q)
      &= \sprod{(I-\bar P_h^*)\sigma}{\tau} + \sprod{d(\bar P_h-I)\sigma}{v} + \sprod{(I-\bar P_h^*)d\sigma}{v}
      \\
      & \quad + \sprod{d(\bar P_h-I)u}{d P_h v}
      + \alpha_h \sprod{(I-\bar P_h^*)(I-\bar P_h)u}{v}
      + \sprod{(I-\bar P_h)u}{q}
      \\
      &\le C (1+\alpha_h)\ve(x) \norm{(\tau,v,q)}_{X_h}
  \end{aligned}
\end{equation}
where 
$\ve(x) := \norm{(I-\bar P_h^*)\sigma} + \norm{d(I-\bar P_h)\sigma} + \norm{(I-\bar P_h^*)d\sigma} + \norm{(I-\bar P_h)u}_V$
satisfies
\begin{equation} \label{boundeps}
  \begin{aligned}
    \ve(x) 
    &\le C \Big(\inf_{\rho \in M^{\ell-1}_h} \norm{\sigma - \rho}
          + \inf_{\rho \in V^{\ell-1,c}_h} \norm{\sigma - \rho}_V
           + \inf_{w \in M^{\ell}_h} \norm{d\sigma - w}
           + \inf_{w \in V^{\ell,c}_h} \norm{u-w}_V \Big)
    \\
    &\le C \eta(x)
 \end{aligned}
\end{equation}
by using \eqref{optP}--\eqref{optP*}, and the definition of $\eta$ in \eqref{eta}.
Since $b_h$ and $\bar b_h$ coincide on $X_h$ this yields an error equation, 
\begin{equation} \label{erreq}
  \bar b_h(x-x_h; \tau, v, q)
  = r(x; \tau, v, q) - \sprod{Q_{\frH^c_h}u}{q}.
\end{equation}
To use the stability Lemma~\ref{lem:bstab_h} we next let 
$
\tilde x^c_h = (\tilde \sigma_h, \tilde u_h, \tilde p_h) \in X^c_h \subset X_h
$
be the $V$-projection of the exact solution $x = (\sigma,u,p)$ onto the discrete 
conforming spaces $V^{\ell-1,c}_h$, $V^{\ell,c}_h$ and $\frH^{\ell,c}_h$:
by optimality of the orthogonal projections this gives
$$
\norm{\tilde \sigma_h - \sigma}_V = \inf_{\rho \in V^{\ell-1,c}_h} \norm{\sigma - \rho}_V
\quad \text{ and } \quad 
\norm{ \tilde u_h - u}_V = \inf_{w \in V^{\ell,c}_h} \norm{u - w}_V.
$$
For the projection error on $\frH^{\ell,c}_h$ we invoke equation~(33) 
from the proof of \cite[Th.~3.9]{Arnold.Falk.Winther.2010.bams},
which reads with the present notation
$$
\norm{\tilde p_h - p} \le C \norm{(I-\pi_h)p} \le C \inf_{w \in V^{\ell,c}_h} \norm{p - w}_V.
$$
This shows that 
$\norm{\tilde \sigma_h - \sigma}_V + \norm{ \tilde u_h - u}_V + \norm{\tilde p_h - p} \le C \eta(x)$.
Using \eqref{X(h)boundX} we next find that
$\norm{\tilde x^c_h - x}_{X(h)} \le C (\norm{\tilde \sigma_h - \sigma}_V + \norm{ \tilde u_h - u}_V + \norm{\tilde p_h - p})$,
so that the previous bound and \eqref{boundeps} yield
\begin{equation} \label{err_tx}
  \norm{\tilde x^c_h - x}_{X(h)} + \ve(x)
  \le C \eta(x).
\end{equation}
Using next the error equation \eqref{erreq}, the continuity \eqref{bbcont} in the extended space $X(h)$ and
the estimate \eqref{rem}, we write for $\tilde x^c_h - x_h \in X_h$
\begin{multline*} 
    b_h(\tilde x^c_h - x_h; \tau, v, q)
    = \bar b_h(\tilde x^c_h - x_h; \tau, v, q)
    = \bar b_h(\tilde x^c_h - x; \tau, v, q) + \bar b_h(x - x_h; \tau, v, q)
    \\
    \le C(1+\alpha_h)\eta(x) \norm{(\tau,v,q)}_{X_h},
\end{multline*}
hence the stability Lemma~\ref{lem:bstab_h} applies. Together with \eqref{err_tx},
it allows us to write
\begin{equation*}
\norm{x - x_h}_{X(h)} \le
  \norm{x - \tilde x^c_h}_{X(h)} +
  \norm{\tilde x^c_h - x_h}_{X_h}
  \le C(1+\alpha_h)\eta(x)
\end{equation*}
where we have used that the norms of $X(h)$ and $X_h$ coincide on the latter space.
This shows the first part of estimate \eqref{errestim}.
To show the second part (and the estimate~\eqref{errestim_unfiltered} in the case
of the unfiltered source problem) we consider the modified (partially conforming) solution
$
\tilde x_h := (\sigma_h, P_h u_h, p_h)
$
and write, in place of \eqref{baserr_bhxh},
\begin{equation} \label{baserr_bhxhc}
  \begin{aligned}
    b_h(\tilde x_h; \tau, P_h v, q)
    &= b_h(x_h; \tau, P_h v, q) - b_h(0,(I-P_h)u_h, 0; \tau, P_h v, q)
    \\
    &= \sprod{f}{P_h v} + \sprod{dP_h\tau}{(I-P_h)u_h},
  \end{aligned}
\end{equation}
which leads to a modified error equation
\begin{equation} \label{erreq_c}
  \bar b_h(x-\tilde x_h; \tau, P_hv, q)
  = r(x; \tau, P_hv, q) - \sprod{Q_{\frH^c_h}u}{q} - \sprod{dP_h\tau}{(I-P_h)u_h}.
\end{equation}
We then observe that for such a partially conforming test function, one may drop the
parameter $\alpha_h$ in the continuity \eqref{bcont} and in the estimate \eqref{rem},
leading to
\begin{equation*} 
  \begin{aligned}
    b_h(\tilde x^c_h - \tilde x_h; \tau, P_h v, q)
    &= \bar b_h(\tilde x^c_h - x; \tau, P_h v, q) + \bar b_h(x - \tilde x_h; \tau, P_h v, q)
    \\
    &\le C\big(\eta(x) + \norm{(I-P_h)u_h}\big) \norm{(\tau,P_h v,q)}_{X_h}
  \end{aligned}
\end{equation*}
where we have also used \eqref{err_tx}.
We then invoke again Lemma~\ref{lem:bstab_h} with the partially conforming
$\tilde x^c_h - \tilde x_h$: this allows us to use a conforming test function $v=P_h v$, hence 
\begin{equation*}
\norm{x - \tilde x_h}_{X(h)} \le
\norm{x - \tilde x^c_h}_{X(h)} + \norm{\tilde x^c_h - \tilde x_h}_{X_h}
  \le C\big( \eta(x) + \norm{(I-P_h)u_h} \big),
\end{equation*}
where \eqref{err_tx} has been used again to bound $\norm{x - \tilde x^c_h}_{X(h)}$.
To complete the proof we observe that
$  \norm{x - x_h}_{X(h)} \le \norm{x - \tilde x_h}_{X(h)} + \norm{(I-P_h)u_h}$
and bound the jump terms with \eqref{jbound_fil} (or Remark~\ref{rem:unfiltered} for the unfiltered problem).
 \end{proof}


\subsection{The Hodge Laplace eigenvalue problem}
\label{sec:brokeigen}

The eigenvalue problem for the Hodge Laplacian operator
is to find $\lambda \in \RR$ and $u \in D(L^\ell)\setminus \{0\}$, solution to
\begin{equation} \label{HL_eig}
  L^\ell u = \lambda u.
\end{equation}
In our broken-FEEC framework we consider its approximation by the CONGA 
operator \eqref{Lha} with positive stabilization parameter $\alpha_h > 0$,
see assumption~\ref{as:ualp}.
The associated problem thus consists of finding 
$\lambda_h \in \RR$ and $u_h \in V^\ell_h \setminus \{0\}$ such that
\begin{equation} \label{HL_eig_h}
  L^\ell_{h,\alpha_h} u_h = \lambda_h u_h.
\end{equation}
In the conforming case the convergence of discrete eigenvalue problems in
the sense of \cite{Boffi.2006.csd} has been established for various discretizations
of the de Rham sequence \cite{Boffi.2000.numa,Monk.Demkowicz.2001.mcomp,Caorsi.Fernandes.Raffetto.2000.sinum},
and for general $L^2$ Hilbert complexes on $s$-regular domains $\Omega$ with $0 < s \le 1$,
see \cite[Sec.~7.7]{Arnold.Falk.Winther.2006.anum}, under
the assumption that the cochain projection $\pi_h$ is uniformly
bounded in $W^\ell = L^2(\Omega)$ and satisfies (assuming now that $h \to 0$ represents a mesh size)
\begin{equation} \label{pi_approx}
  \norm{(I-\pi_h)v} \le C h^s \norm{v}_{H^s}, \quad \text{ for } v \in H^s(\Omega), ~ 0 \le s \le 1
\end{equation}
with a constant independent of $h$.
Then, it is shown that the solution to the
discrete conforming problem \eqref{HL_mixed_hc} satisfies the refined error estimate
\begin{equation}\label{refest_c}
  \norm{\sigma-\sigma^c_h} + \norm{u - u^c_h} + \norm{du -du^c_h} + \norm{p-p^c_h}
  \le C h^s \norm{f}
\end{equation}
see \cite[Th.~7.10]{Arnold.Falk.Winther.2006.anum}. The convergence of the eigenvalue
problem follows by applying arguments from the perturbation theory of linear operators.
Specifically, introducing the solution operators
$K : f \to u + p$ and $K^c_h : f \to u^c_h + p^c_h$ associated
with the continuous and discrete source problems,
\eqref{HL_b_form} and \eqref{HL_mixed_hc},
estimate~\eqref{refest_c} yields 
\begin{equation} \label{conv_Khc}
  \norm{K - K^c_h}_{\cL(L^2, L^2)} \le C h^s
\end{equation}
so that the convergence $K^c_h \to K$ holds in $L^2$ operator norm as $h \to 0$,
which itself is a necessary and sufficient condition for the convergence of the eigenvalue problem,
see \cite{Boffi.Brezzi.Gastaldi.2000.mcomp} or \cite[Sec.~8.3]{Arnold.Falk.Winther.2006.anum}.

To establish a similar result for our CONGA Hodge Laplacian operator
we consider the solution operator $K_h : f \to u_h + p_h$ associated
with the unfiltered nonconforming source problem \eqref{HLh_b_unfiltered}.
We also complete assumption \eqref{pi_approx} with a similar approximation property for
the conforming projection and its adjoint, namely
\begin{equation} \label{PP*_approx}
  \norm{(I - \bar P_h)v}, \norm{(I-\bar P^*_h)v} \le C h^s \norm{v}_{H^s(\Omega)} \quad \text{ for } ~ 0 \le s \le 1
\end{equation}
which, given \eqref{optP}--\eqref{optP*}, essentially amounts to a
first-order approximation property for the spaces $V^{\ell,c}_h$ and $M^{\ell}_h$.
We then have the following refined estimate.
\begin{theorem} \label{th:refest}
  Assume that the domain $\Omega$ is $s$-regular for some $0 < s \le 1$, and
  that the projection operators satisfy \eqref{pi_approx} and \eqref{PP*_approx}.
  Then for a penalization parameter such that
  \begin{equation} \label{alphacv}
    \alpha_h \ge C h^{-s}
  \end{equation}
  the solutions to the continuous and unfiltered discrete
  problems \eqref{HL_b_form}, \eqref{HLh_b_unfiltered} satisfy
  \begin{equation}\label{refest}
    \norm{\sigma-\sigma_h} + \norm{u - u_h} + \norm{du -dP_hu_h} + \norm{p-p_h}
    \le C h^s \norm{f}
  \end{equation}
  with a constant independent of $h$.
\end{theorem}
\begin{remark}
  This refined error estimate (and the proof below) also apply to the
  filtered source problem \eqref{HLh_b_form}.
\end{remark}
\begin{proof}[Proof.~]
  To bound the error we will use \eqref{refest_c} and estimate the quantity
  \begin{equation*} 
    \norm{\sigma^c_h-\sigma_h} + \norm{u^c_h-u_h} + \norm{du^c_h-dP_hu_h} + \norm{p^c_h-p_h}.
  \end{equation*}  
  We start by decomposing the solution as in \eqref{uh_dec}, 
  and we use \eqref{HL_mixed_h_B*} and \eqref{p_uH}
  (valid for both the filtered and the unfiltered
  problems, see Remark~\ref{rem:unfiltered}):
  this yields $p_h = p^c_h$, as well as
  $$
  P_h u_h - u_\frB = u_\frH + u_{\frB^*}
  = u^c_\frH + u^c_{\frB^*} = u^c_h - u^c_\frB,
  $$
  that is, $u^c_h - P_h u_h = u^c_\frB-u_\frB$.
  This readily gives
  $du^c_h - dP_h u_h = 0$,
  so that we have
  \begin{equation}\label{prerefest}
  \norm{u^c_h - u_h} + \norm{du^c_h -dP_hu_h} + \norm{p^c_h-p_h}
  \le
  \norm{u^c_\frB - u_\frB} + \norm{(I-P_h)u_h}.
  \end{equation}
  For the latter term we use the last bound in \eqref{jbound_fil} (and Remark~\ref{rem:unfiltered}),
  together with the penalization scaling \eqref{alphacv}: this yields
  \begin{equation} \label{jbound}
    \norm{(I-P_h)u_h} \le \alpha_h^{-1} \norm{f} \le C h^s \norm{f}.
  \end{equation}
  Hence, it remains to study the errors in $u_\frB$ and $\sigma_h$. 
  For this we remind systems \eqref{HL_mixed_hc_B} and \eqref{HL_mixed_h_B}
  which correspond to test functions $v$ in $\frB^{\ell,c}_h = \frB^\ell_h$: they read
  \begin{equation} \label{sys_B} 
    \left\{\begin{aligned}
     \sprod{\sigma^c_h}{\tau^c} &= \sprod{d\tau^c}{u^c_\frB} 
     \\
     \sprod{d \sigma^c_h}{v} &= \sprod{f}{v} 
    \end{aligned} \right.
   \qquad \text{ and } \qquad
     \left\{\begin{aligned}
      \sprod{\sigma_h}{\tau} - \sprod{dP_h \tau}{u_h} &= 0 
      \\
      \sprod{dP_h \sigma_h}{v} &= \sprod{f}{v}  
    \end{aligned} \right.
   \end{equation}
   for all $\tau^c \in V^{\ell-1,c}_h$, $\tau \in V^{\ell-1}_h$ and all $v \in \frB^{\ell,c}_h$.
   These systems characterize
   $\sigma^c_h \in \frB^{*,c}_{\ell,h}$ 
   by the relation 
   $d \sigma^c_h = Q_{\frB_h} f$ 
   and similarly 
   $\sigma_h \in \frB^{*}_{\ell,h}$ 
   by the relation 
     $dP_h \sigma_h = Q_{\frB_h} f$.
   This shows in particular that $dP_h(\sigma^c_h-\sigma_h) = 0$,
   so that using 
   $\tau^c = \sigma^c_h-P_h \sigma_h$ in the conforming system yields
   \begin{equation} \label{Psh_sc}
     \sprod{\sigma^c_h}{\sigma^c_h-P_h\sigma_h} = 0,
   \end{equation}
   while taking 
   $\tau = \sigma^c_h-\sigma_h$ in the nonconforming one
   gives
   \begin{equation} \label{sh_sc}
     \sprod{\sigma_h}{\sigma^c_h-\sigma_h} = 0.
   \end{equation}
   We then use successively \eqref{sh_sc} and \eqref{Psh_sc} to compute
   \begin{equation*}
     \begin{aligned}
       \norm{\sigma^c_h-\sigma_h}^2 
       &= \sprod{\sigma^c_h}{\sigma^c_h-\sigma_h}
       = \sprod{\sigma^c_h}{(P_h-I)\sigma_h}
       \\
       &= \sprod{\sigma}{(P_h-I)\sigma_h} + \sprod{\sigma^c_h-\sigma}{(P_h-I)\sigma_h}
       \\
       &= \sprod{\sigma}{(P_h-I)(\sigma_h-\sigma^c_h)} + \sprod{\sigma^c_h-\sigma}{(P_h-I)(\sigma_h-\sigma^c_h)}
     \end{aligned}    
   \end{equation*}
   where the last step uses $(P_h-I)\sigma^c_h = 0$ which follows from $\sigma^c_h \in V^{\ell,c}_h$.
   We handle the second term by using 
   the stability of $P_h$ and the estimate \eqref{refest_c}: 
   \begin{equation*}
     \sprod{\sigma^c_h-\sigma}{(P_h-I)(\sigma_h-\sigma^c_h)} 
     \le \norm{\sigma^c_h-\sigma}\norm{(P_h-I)(\sigma_h-\sigma^c_h)}
     \le C h^s \norm{f}\norm{\sigma_h-\sigma^c_h}.
   \end{equation*}
   We next use our assumption that $P_h$ can be extended by $\bar P_h$ on $W^\ell$, with adjoint $\bar P^*_h$:
   this allows us to write
   \begin{equation*}
     \sprod{\sigma}{(P_h-I)(\sigma_h-\sigma^c_h)} 
     = \sprod{(\bar P^*_h-I)\sigma}{\sigma_h-\sigma^c_h} 
     \le \norm{(\bar P^*_h-I)\sigma}\norm{\sigma_h-\sigma^c_h}.
   \end{equation*}
   Since $\Omega$ is $s$-regular, 
   \cite[Eq.~(7.29)]{Arnold.Falk.Winther.2006.anum} gives us an a priori bound
   $\norm{\sigma}_{H^s(\Omega)} \le C \norm{f}$,
   so that \eqref{PP*_approx} yields
   \begin{equation} \label{sigma_approx_2}
     \norm{(I-\bar P_h^*)\sigma} \le C h^s \norm{f}.
   \end{equation}
   Gathering the above bounds yields
     $\norm{\sigma^c_h-\sigma_h}^2 \le Ch^s\norm{f} \norm{\sigma_h-\sigma^c_h}$,
   which leads to the desired estimate for the $\sigma$ error, namely
   \begin{equation} \label{est-s}
     \norm{\sigma^c_h-\sigma_h} \le Ch^s\norm{f}.
   \end{equation}
  Turning to the error in $u_\frB$, we now take $\tau = \tau^c \in V^{\ell-1,c}_h$ in \eqref{sys_B}: this yields
  \begin{equation} \label{sch-sh}  
    \sprod{\sigma^c_h - \sigma_h}{\tau^c} = \sprod{d\tau^c}{u^c_\frB - u_h} = \sprod{d\tau^c}{u^c_\frB - Q_{\frB_h} u_h} 
      \quad \forall \, \tau^c \in V^{\ell-1,c}_h
  \end{equation}
  where the second equation follows from the fact that $d\tau^c \in \frB_h$.
  We then let $\tau^c \in \frB^{*,c}_{\ell-1,h}$ be such that
  $d\tau^c = u^c_\frB - Q_{\frB_h}u_h \in \frB^{\ell,c}_h$.
  This allows us to compute
  \begin{equation*}
    \begin{aligned}
      \norm{ u^c_\frB - Q_{\frB_h} u_h}^2
      &= \sprod{d\tau^c}{u^c_\frB - Q_{\frB_h} u_h}
      = \sprod{\sigma^c_h-\sigma_h}{\tau^c}
      \le \norm{\sigma^c_h-\sigma_h}\norm{\tau^c} 
      \\
      &\le \cph \norm{\sigma^c_h-\sigma_h}\norm{d\tau^c} = \cph \norm{\sigma^c_h-\sigma_h}\norm{u^c_\frB - Q_{\frB_h} u_h}
    \end{aligned}
  \end{equation*}
  where the second equality is \eqref{sch-sh} and the second bound 
  is the discrete Poincaré inequality \eqref{pcr_hc}.
  Using Assumption~\ref{as:bcp} and the bound \eqref{est-s}, this gives
  \begin{equation} \label{bound_uB2}
    \norm{ u^c_\frB - Q_{\frB_h} u_h} \le C\norm{\sigma^c_h-\sigma_h} \le  Ch^s\norm{f}.
  \end{equation}
  Writing next  $u_{\frB} = Q_{\frB_h} P_h u_h = Q_{\frB_h} u_h + Q_{\frB_h} (P_h-I) u_h$, we then estimate
  \begin{equation} \label{bound_uB}
  \norm{u^c_{\frB} - u_{\frB}} \le \norm{u^c_\frB - Q_{\frB_h} u_h} + \norm{Q_{\frB_h}(I-P_h)u_h}
  \le Ch^s\norm{f} 
  \end{equation}
  where the second inequality follows from \eqref{bound_uB2}, \eqref{jbound} and the fact that 
  $Q_{\frB_h}$ is stable (actually of unit norm) in $W^\ell$. 
  The desired result follows by gathering estimates \eqref{refest_c}, \eqref{prerefest}, \eqref{jbound}, \eqref{est-s} and \eqref{bound_uB}.

\end{proof}

As a consequence of the above result, we find that the solution operator $K_h$
for the unfiltered source problem converges in operator norm towards $K$,
\begin{equation} \label{conv_Kh}
  \norm{K - K_h}_{\cL(L^2, L^2)} \le C h^s.
\end{equation}
Invoking the same arguments as above this leads to the following result.

\begin{theorem} \label{th:spectral}
  Under the same assumptions as Th.~\ref{th:refest}, the discrete eigenvalue
  problem \eqref{HL_eig_h} converges to the continuous problem \eqref{HL_eig} in the sense
  of \cite[Def.~2.1]{Boffi.2006.csd}.
\end{theorem}

We remind that this result guarantees that all the exact eigenvalues
with their respective eigenspaces are well approximated as $h \to 0$,
which in particular means that the exact multiplicies are preserved at convergence,
and that the discrete operator is free of spurious eigenvalues.

\section{Application to polynomial finite elements}
\label{sec:cartfem}

In this section we describe how the above theory applies to tensor-product polynomial finite elements 
on a 2D Cartesian domain $\Omega = ]0,a[ \times ]0,a[$. 
Application to unstructured elements is also possible following the same lines as in 
\cite{Campos-Pinto.Sonnendrucker.2017a.jcm,Campos-Pinto.Sonnendrucker.2017b.jcm} 
and we refer to \cite{conga_psydac} for an application to mapped spline elements involving 
multiple patches on complex domains, where the CONGA operators are used to approximate 
various electromagnetic problems.

\subsection{Tensor-product local spaces}
\label{sec:tensor}

We partition the domain $\Omega$ into a collection of Cartesian cells
$\Omega_{\bk} = ]h (k_1-1), h k_1[ \, \times \, ]h (k_2-1), h k_2[$ of step size $h = a/K$
with $k_1, k_2 = 1, \dots, K$,
and on each cell we consider a local de Rham sequence 
of tensor-product polynomial spaces, of the form
\begin{equation} \label{derham_loc}
  \VV^0(\Omega_\bk) \xrightarrow{ \mbox{$~ \bgrad ~$}}
    \VV^1(\Omega_\bk) \xrightarrow{ \mbox{$~ \curl ~$}}
       \VV^2(\Omega_\bk)
\end{equation}
with $\curl v = \partial_1 v_2 - \partial_2 v_1$ the scalar curl in 2D, and local spaces defined as
\begin{equation*}
\VV^0(\Omega_\bk) = \QQ_{p,p}(\Omega_\bk)
, \quad  
\VV^1(\Omega_\bk) = \begin{pmatrix} \QQ_{p-1,p}(\Omega_\bk) \\ \QQ_{p,p-1}(\Omega_\bk) \end{pmatrix}
, \quad 
\VV^2(\Omega_\bk) = \QQ_{p-1,p-1}(\Omega_\bk)
\end{equation*}
where
$
\QQ_{p_1,p_2} := \Span \big\{ x_1^{r_1} x_2^{r_2} : r_d = 0, \dots, p_d \big\}
$.
The global broken spaces are then defined as the Cartesian product of the local spaces 
\begin{equation} \label{Vh}
V^{\ell}_h := \VV^\ell(\Omega_{(1,1)}) \times \cdots \times \VV^\ell(\Omega_{(K,K)})
\end{equation}
which can also be seen as a sum of spaces if one extends the local spaces by 0 outside of their cell.
For the global conforming spaces we consider
$
V^{\ell,c}_h := V^{\ell}_h \cap V^\ell
$
with $V^\ell$ defined as the usual Hilbert spaces of the 2D grad-curl de Rham sequence
with homogeneous boundary conditions, namely
\begin{equation} \label{derham_2D}
V^{0} = H^1_0(\Omega) \xrightarrow{ \mbox{$~ \bgrad ~$}}
  V^{1} = H_0(\curl;\Omega) \xrightarrow{ \mbox{$~ \curl ~$}}
      V^{2} = L^2(\Omega).
\end{equation}
Here the $W^\ell$ spaces are 
\begin{equation} \label{Well_2D}
W^{0} = L^2(\Omega)~,
  \qquad W^{1} = L^2(\Omega)^2~,
      \qquad W^{2} = L^2(\Omega)
\end{equation}
and the dual sequence \eqref{seqdual} is devoid of boundary conditions: it reads
\begin{equation} \label{derham_dual}
V^{*}_0 = L^2(\Omega) \xleftarrow{ \mbox{$~ -\Div ~$}}
  V^*_{1} = H(\Div;\Omega) \xleftarrow{ \mbox{$~ \bcurl ~$}}
      V^*_{2} = H(\bcurl;\Omega)
\end{equation}
where $\bcurl v = (\partial_2 v, -\partial_1 v)$ is the vector-valued curl operator in 2D.

\subsection{Geometric degrees of freedom}
\label{sec:geodofs}

On each local space we consider degrees of freedom corresponding to the
interpolation / histopolation approach of~\cite{Robidoux.2008.histo, %
Gerritsma.2011.spec, kreeft2011mimetic} and also used in
recent plasma-related applications~\cite{CPKS_variational}.
We equip every interval $I_k = ]h(k-1), hk[$ with a Gauss-Lobatto grid 
$
h(k-1) = \zeta_{k,0} < \cdots < \zeta_{k,p} = hk,
$
leading to subgrids of the cells $\Omega_\bk$ 
made of nodes, (small) edges and subcells:
\begin{equation} \label{subgrid}
  \left\{ \begin{aligned}
    &\ttn_{\bk,\bi} := (\zeta_{k_1, i_1}, \zeta_{k_2, i_2})
    \quad &&\text{ for }
    (\bk, \bi) \in \cM^0_h 
    \\
    &\tte_{\bk,d,\bi} := [\ttn_{\bk,\bi}, \ttn_{\bk,\bi+\uvec_d}]
    \quad &&\text{ for }
    (\bk, d, \bi) \in \cM^1_h
    \\
    &\ttc_{\bk,\bi} := \tte_{\bk,1,\bi} \times \tte_{\bk,2,\bi}
    \quad &&\text{ for }
    (\bk, \bi) \in \cM^2_h~.
  \end{aligned} \right.
\end{equation}
Here the square brackets $[\cdot]$ denote a convex hull, $\uvec_d$ is the canonical
basis vector of $\RR^2$ along dimension $d \in \{1,2\}$, and we have used the multi-index sets 
\begin{equation*}
  \left\{ \begin{aligned}
  & \cM^0_h := \{(\bk, \bi) : \bk \in \range{1}{K}^2, \bi \in \range{0}{p}^2 \}
  \\
  & \cM^1_h := \{(\bk, d, \bi) : \bk \in \range{1}{K}^2, d \in \{1,2\}, \bi \in \range{0}{p}^2, i_d < p\}
  \\
  & \cM^2_h := \{(\bk, \bi) : \bk \in \range{1}{K}^2, \bi \in \range{0}{p-1}^2 \}.
  \end{aligned} \right.
\end{equation*}
On these geometrical elements we define pointwise and integral degrees of freedom,
\begin{equation} \label{geodofs}
    \left\{ \begin{aligned}
        &\sigma^{0}_{\bk,\bi}(\vp) :=
            (\vp|_{\Omega_\bk})(\ttn_{\bk,\bi})
        \quad &&\text{ for }
        (\bk, \bi) \in \cM^0_h
        \\
        &\sigma^{1}_{\bk,d,\bi}(\bv) :=
            \int_{\tte_{\bk,d,\bi}}\uvec_d \cdot (\bv|_{\Omega_\bk})
        \quad &&\text{ for }
        (\bk, d, \bi) \in \cM^1_h
        \\
        &\sigma^{2}_{\bk,\bi}(\rho) := \int_{\ttc_{\bk,\bi}}\rho|_{\Omega_\bk}
        \quad &&\text{ for }
        (\bk, \bi) \in \cM^2_h~.
    \end{aligned} \right.
\end{equation}

\subsection{Broken basis functions}

The local spaces are then spanned with tensor-products of univariate
polynomials: on an arbitrary interval we let
$
\phi_{k,i} \in \PP_p(I_k) 
$
be the interpolation (Lagrange) polynomials associated with the Gauss-Lobatto nodes, 
\begin{equation*}
\phi_{k,i}(\zeta_{k,j}) = \delta_{i,j} , \quad i,j = 0, \dots, p,
\end{equation*}
and we let
$
\psi_{k,i} \in \PP_{p-1}(I_k) 
$
be the {\em histopolation} polynomials associated with the Gauss-Lobatto sub-intervals, 
which are characterized by the relations
\begin{equation*}
\int_{\zeta_{k,j}}^{\zeta_{k,j+1}} \psi_{k,i} = \delta_{i,j} , \quad i,j = 0, \dots, p-1.
\end{equation*}
The local tensor-product basis functions
are then defined as 
\begin{equation} \label{bf}
\left\{\begin{aligned}
&\Lambda^0_{\bk,\bi}(\bx) := \phi_{k_1,i_1}(x_1) \phi_{k_2,i_2}(x_2) 
    \quad && \text{ for }
    (\bk, \bi) \in \cM^0_h
\\
&\Lambda^1_{\bk,d,\bi}(\bx) := \uvec_d \psi_{k_d,i_d}(x_d)
    \phi_{k_{d'},i_{d'}}(x_{d'}) 
    \quad && \text{ for }
    (\bk, d, \bi) \in \cM^1_h, ~ d' = 3-d
\\
&\Lambda^2_{\bk,\bi}(\bx) := \psi_{k_1,i_1}(x_1) \psi_{k_2,i_2}(x_2)
  \quad && \text{ for }
  (\bk, \bi) \in \cM^2_h.
\end{aligned}\right.
\end{equation}
As such, they provide a basis for the respective broken spaces $V^\ell_h$
which is dual to the degrees of freedom~\eqref{geodofs}, i.e.
$
\sigma^\ell_{\bsm}\bigl(\Lambda^\ell_{\bn}\bigr) = \delta_{\bsm, \bn}\ \forall \bsm, \bn \in \cM_h^\ell
$\
for $\ell = 0, 1, 2$.
They also allow the derivation of simple expressions for the
conforming projection operators~$P^\ell_h$.

\subsection{Conforming projection} 

In this Cartesian setting, it is easy to verify that
a function $v$ in the space $V^\ell_h$
belongs to the conforming subspace $V^{\ell,c}_h$
if it possesses the proper continuity across cell interfaces, namely
if it is continuous for $\ell = 0$, and if its tangential traces are continuous for $\ell=1$.
Given the form of the geometric degrees of freedom \eqref{geodofs}, these continuity conditions
may be expressed by the equality of the coefficients associated with a single geometrical element.
Specifically, let us denote by $\ttg^\ell_\mu$ the geometrical element of dimension $\ell$
associated with the multi-index $\mu \in \cM^\ell_h$.
Then $v = \sum_{\mu \in \cM^\ell_h} v_\mu \Lambda^\ell_\mu$ belongs to $V^{\ell,c}_h$
if the continuity conditions are satisfied, i.e.,
\begin{equation*}
  v_\mu = v_\nu
  \qquad \text{ for all $\mu, \nu$ 
    such that $\ttg^\ell_{\mu} = \ttg^\ell_{\nu}$}
\end{equation*}
and the homogeneous boundary conditions are satisfied if in addition we have
\begin{equation*}
  v_\mu = 0
  \qquad \text{ for all $\mu$ 
  such that $\ttg^\ell_{\mu} \in \partial \Omega$}.
\end{equation*}
In particular, a natural basis for the conforming space $V^{\ell,c}_h$ is obtained by stitching
together the broken basis functions associated with a single interior geometrical element,
and by discarding the boundary ones. Gathering the former in the sets
$ 
  \ttG^\ell_h := \{\ttg^\ell_\mu : \mu \in \cM^\ell_h, \ttg^\ell_\mu \not\in \partial \Omega\},
$ 
the resulting conforming basis functions read
\begin{equation} \label{bf_c}
\Lambda^{\ell,c}_\ttg := \sum_{\mu \in \cM^\ell_h(\ttg)} \Lambda^{\ell}_\mu \quad \text{ for } ~ \ttg \in \ttG^\ell_{h}
\end{equation}
where
$ 
  \cM^\ell_h(\ttg) := \{\mu \in \cM^\ell_h : \ttg^\ell_\mu = \ttg\}
$ 
denote the multi-indices associated with a given geometrical element.
A simple conforming projection then consists of an averaging
\begin{equation} \label{Pell}
P^\ell_h \Lambda^\ell_\nu :=
\frac{1}{\#\cM^\ell_h(\ttg^\ell_\nu)} \sum_{\mu \in \cM^\ell_h(\ttg^\ell_\nu)}  \Lambda^{\ell}_{\mu}.
\end{equation}
It is easily verified that this defines a projection $P^\ell_h: V^\ell_h \to V^\ell_h$ 
onto the conforming subspace $V^{\ell,c}_h \subset V^\ell_h$, 
with matrix entries (assuming some implicit numbering of the degrees of freedom)
\begin{equation*}
\matP^\ell_{\mu,\nu}
:= \sigma^{\ell}_\mu(P^\ell_h \Lambda^\ell_\nu)
= \begin{cases}
  \big(\#\cM^\ell_h(\ttg^\ell_\nu)\big)^{-1} \quad &\text{ if } \ttg^\ell_\mu = \ttg^\ell_\nu
  \\
  0 \quad &\text{ otherwise }
\end{cases},
\qquad \text{ for } \mu, \nu \in \cM^\ell_h.
\end{equation*}
Denoting by
\begin{equation} \label{matM}
  \matM^\ell := \Big(\sprod{\Lambda^{\ell}_\mu}{\Lambda^{\ell}_\nu})\Big)_{\mu,\nu \in \cM^\ell_h}
\end{equation}
the mass matrix in the broken basis of $V^\ell_h$
(remind that $\sprod{\cdot}{\cdot}$ is the $L^2$ scalar product in the 
proper $W^\ell$ space, see \eqref{Well_2D}), we then find that 
the matrix of the adjoint projection $(P^\ell_h)^* : V^\ell_h \to V^\ell_h$ takes the form
\begin{equation} \label{matP*}
\Big(\sigma^{\ell}_\mu((P^\ell_h)^* \Lambda^\ell_\nu)\Big)_{\mu \in \cM^{\ell}_h, \nu \in \cM^{\ell+1}_h} = 
(\matM^{\ell})^{-1} (\matP^\ell)^T \matM^{\ell}.
\end{equation}

An attracting feature of Gauss-Lobatto nodes is that this simple averaging procedure 
automatically yields moment preservation.
\begin{lemma}
    The conforming projection defined by \eqref{Pell} preserves polynomial moments of degree $p-1$
    with homogeneous boundary conditions. Namely, for all $v \in V^\ell_h$ it holds
    \begin{equation} \label{mp}
      \sprod{P^\ell_h v }{\psi} = \sprod{v}{\psi}  \qquad \text{for all } \psi \in \QQ_{p-1,p-1}(\Omega) \cap V^\ell.
    \end{equation}
\end{lemma}
\begin{remark}
  If the spaces $V^\ell$ (and the conforming subspaces $V^{\ell,c}_h$) are defined without boundary conditions, then the result
  \eqref{mp} holds in the same form.
\end{remark}

\begin{proof}
  Given the tensor-product structure, it is enough to verify that the property holds in the univariate
  case ($\Omega = ]0,a[$) where the only nontrivial projection is $P^0_h$, defined on $V^0_h$ the space of 
  piecewise polynomials of degree $p$, onto its continuous subspace with boundary condition $V^{0,c}_h = V^0_h \cap H^1_0(\Omega)$.
  Thus, for $v \in V^0_h$ written in the form $v= \sum_{k,i} v_{k,i} \phi_{k,i}$
  and $\psi$ a polynomial of degree $p-1$, we observe that the Gauss-Lobatto quadrature formulas
  are exact on any cell:
  \begin{equation} \label{GL}
    \int_{\Omega_k} v \psi 
      = \sum_{i,j=0}^p \omega_{j} v_{k,i} \phi_{k,i}(\zeta_{k,j}) \psi(\zeta_{k,j}) 
      = \sum_{i=0}^p \omega_{i} v_{k,i} \psi(\zeta_{k,i}).
  \end{equation}
  From \eqref{Pell} we have
  $(P^0_h v)_{k,p} = (P^0_h v)_{k+1,0} := \frac 12 (v_{k,p}+v_{k+1,0})$ 
  for every cell vertex $\zeta_{k,p} = hk$ with $1 \le k < K$, 
  $(P^0_h v)_{k,i} := 0$ on boundary vertices $\zeta_{k,i} \in \partial\Omega$ 
  and $(P^0_h v)_{k,i} := v_{k,i}$ on every other node.
  Applying \eqref{GL} to $v \leftarrow P^0_h v-v$, using the weights symmetry $\omega_p = \omega_0$ 
  and the boundary condition $\psi = 0$ on $\partial \Omega$, we thus find
  $$
  \int_{\Omega} (P^0_h v -v) \psi 
  = \sum_{1 \le k < K} \omega_0 \big((P^0_h v)_{k,p} - v_{k,p} + (P^0_h v)_{k+1,0} - v_{k+1,0}\big) \psi(hk) = 0,
  $$
  hence the claim. 
\end{proof}  

According to the error analysis in section~\ref{analysis}, this moment-preserving property makes $P^\ell_h$ 
a good candidate for a CONGA scheme of order $p$. We do not know, however, if the conforming projections \eqref{Pell}
can be extended by $V$-stable projection operators $\bar P^\ell_h$ on $V^\ell$, so that we question as to 
whether our analysis applies here remains open for the time being.

\subsection{Differential operators in matrix form} \label{sec:DP_mat}

Using the broken degrees of freedom \eqref{geodofs} and basis functions \eqref{bf},
we let $\matD^\ell$ be the matrix of
the piecewise differential operator defined on each cell as
$d^\ell_{\rm pw}:= d^\ell: \VV^\ell(\Omega_\bk) \to \VV^{\ell+1}(\Omega_\bk)$, $\bk \in \range{1}{K}^2$,
\begin{equation} \label{matD}
  \matD^\ell_{\mu,\nu} := \sigma^{\ell+1}_\mu(d^\ell_{\rm pw} \Lambda^\ell_\nu)
  \qquad \text{ for } ~ \mu \in \cM^{\ell+1}_h, \nu \in \cM^\ell_h,
\end{equation}
(on conforming functions $v \in V^\ell$ this is just the usual differential $d^\ell$).
By construction, both $\matD^\ell$ and $\matM^\ell$ have a cell-diagonal structure
in the sense that they do not couple different cells.
Moreover, the geometric nature of the degrees of freedom \eqref{geodofs}
leads to differential matrices \eqref{matD} which are connectivity matrices,
i.e. they are composed of $0$ or $\pm 1$ entries
corresponding to the connectivity of the subgrids,
see e.g.~\cite{Gerritsma.2011.spec, kreeft2011mimetic, CPKS_variational}.
Observing that the piecewise and global differential operators coincide on conforming functions,
we find that the matrix of the CONGA differential operator
$d^\ell_h = d^\ell P^\ell_h = d^\ell_{\rm pw} P^\ell_h: V^{\ell}_h \to V^{\ell+1}_h$ reads
\begin{equation} \label{matDP}
  \Big(\sigma^{\ell+1}_\mu(d^\ell_h \Lambda^{\ell}_\nu)\Big)_{\mu \in \cM^{l+1}_h, \nu \in \cM^\ell_h}
  = \matD^\ell \matP^\ell~.
\end{equation}
This matrix is local (two entries can only be connected if they belong to adjacent cells) since 
$\matD^\ell$ and $\matP^\ell$ are respectively cell-diagonal and local (in the same sense).

It is a key feature of the broken FEEC discretization that not only the primal differential operators $d^\ell_h$ are local, but {\em also the dual ones} $d^*_{l+1,h} = (d^\ell_h)^*: V^{\ell+1}_h \to V^\ell_h$.
This is easily seen by writing their matrix,
\begin{equation} \label{matDP*}
  \Big(\sigma^{\ell}_\mu(d^*_{\ell+1,h} \Lambda^{\ell+1}_\nu)\Big)_{\mu \in \cM^{\ell}_h, \nu \in \cM^{\ell+1}_h}
  = (\matM^\ell)^{-1} (\matD^\ell \matP^\ell )^T \matM^{\ell+1}
\end{equation}
and by observing that its locality is the same as that of $\matD^\ell \matP^\ell$, 
since the broken mass matrices are cell-diagonal.
In addition, the dual commuting projection operators \eqref{tpih} are also local.

We emphasize that this is in general not the case with conforming finite elements: 
indeed the matrix of the primal (strong) differential operator \eqref{dlc},
\begin{equation} \label{matDc}
  \matD^{\ell,c} := \Big(
    \sigma^{\ell+1,c}_{\ttg'}(d^\ell \Lambda^\ell_{\ttg})
  \Big)_{\ttg' \in \ttG^{\ell+1}_h, \ttg \in \ttG^{\ell}_h}
\end{equation}
is also local, but this is no longer the case for that of the weak codifferential \eqref{wdlc}, which reads
\begin{equation} \label{matwDc}
  \Big(\sigma^{\ell,c}_{\ttg}(d^{*,c}_{\ell+1,h} \Lambda^{\ell+1,c}_{\ttg'})\Big)_{\ttg \in \ttG^{\ell}_h, \ttg' \in \ttG^{\ell+1}_h}
  = (\matM^{\ell,c})^{-1} (\matD^{\ell,c} )^T \matM^{\ell+1,c}.
\end{equation}
Here $\sigma^{\ell,c}_\ttg$, $\ttg \in \ttG^\ell_h$, 
denote degrees of freedom associated with the conforming basis functions \eqref{bf_c},
(for instance the geometric ones \eqref{geodofs} attached to a single cell $\Omega_k$ 
per geometrical element $\ttg$) and
$$
\matM^{\ell,c} = \Big(
  \sprod{\Lambda^{\ell,c}_\ttg}{\Lambda^{\ell,c}_{\ttg'}})
  \Big)_{\ttg,\ttg' \in \ttG^\ell_h}
$$
is the mass matrix in the conforming space. Since this matrix is local but not block-diagonal in general,
it has no local inverse, which results in \eqref{matwDc} being dense. 
We point out that for structured meshes or low order elements, lumping methods 
based on local quadrature rules do exists which allow one to 
derive local approximations of the inverse mass matrices, 
see e.g. \cite{Cohen_Monk_1998_nmpde,Egger_Radu_2021_sinum}, as well as 
local dual differential operators \cite{Lee_Winther_2018_mcomp,Lee_2022_m2an}.
However it is not clear yet how to extend these methods to high order elements on unstructured or curvilinear cells.
We summarize the above observations as follows.

\begin{theorem} \label{th:loc_bfeec}
  The primal and dual discrete differential operators are local in the sense that their matrices \eqref{matDP} and \eqref{matDP*} only couple degrees of freedom belonging to adjacent cells.
  The dual commuting projection operator \eqref{tpih} is also local, 
  in the sense that the values of $\tilde \pi^\ell_h f$ in a given cell only 
  depend on the values of $f$ in the adjacent cells.  
\end{theorem}

\begin{remark}
  These properties follow from the locality of the conforming projections
  and the broken nature of the discrete spaces. In particular they 
  would also hold on broken FEEC unstructured elements 
  with conforming projections designed following the method 
  of \cite{Campos-Pinto.Sonnendrucker.2016.mcomp}.
\end{remark}
  
\begin{proof}
  As discussed above the locality of the primal differential operators follow directly
  from that of the conforming projections, and that of the dual ones follow from 
  the cell-diagonal structure of the mass matrices. To show the locality of the 
  dual commuting projection, let us denote its coefficients vector by
  $\arr{f} := (\sigma^\ell_\mu(\tilde \pi^\ell_h f))_{\mu \in \cM^\ell_h}$,
  where $f \in W^\ell$ belongs to the proper $L^2$ space according to \eqref{Well_2D}. 
  Using \eqref{tpih_char} we find that
  \begin{equation} \label{arrf_tarrf}
  \arr{f} = (\matM^\ell)^{-1} (\matP^\ell)^T \tilde {\arr{f}}
  \quad \text{ where } \quad
    \tilde {\arr{f}} = (\sprod{\Lambda^\ell_\mu}{f})_{\mu \in \cM^{\ell}_h}.
  \end{equation}
  The result follows again from the fact that $\matM^\ell$ (and its inverse) are 
  cell-diagonal, and that $\matP^\ell$ only couples adjacent cells. 
\end{proof}

\subsection{Hodge Laplace problem in matrix form} \label{sec:HL_mat}

Denoting by $\Lambda^{\ell,\frH}_j$, 
$j = 1, \dots, N^{\ell,\frH}_h$ 
a basis of the discrete harmonic space $\frH^{\ell,c}_h$ 
(which in practice can be computed by solving the equation
$L^\ell_{h,\alpha}v = 0$ in the broken space $V^\ell_h$
for an arbitrary $\alpha > 0$, see Th.~\ref{th:HL_ker})
we may introduce the rectangular (reduced) mass matrix
\begin{equation} \label{mat_MH}
  \matM^{\ell,\frH}
    := \Big(\sprod{\Lambda^{\ell}_\mu}{\Lambda^{\ell,\frH}_j}\Big)_{%
    \mu \in \cM^{\ell}_h, j = 1, \dots, N^{\ell,\frH}_h}
\end{equation}
to describe the discrete harmonic fields.
Using the matrices described in the above sections and letting
\begin{equation} \label{mat_S}
  \matS^\ell := (\matI - \matP^\ell)^T \matM^{\ell} (\matI - \matP^\ell)
  = \Big(
    \sprod{(I-P^\ell_h)\Lambda^{\ell}_\mu}{(I-P^\ell_h)\Lambda^{\ell}_{\nu}})
    \Big)_{\mu,\nu \in \cM^\ell_h}
\end{equation}
denote the jump stabilization matrix in $V^\ell_h$, we then find that the
broken FEEC Hodge Laplacian source problem \eqref{HL_mixed_h}
takes the following matrix form:
\begin{equation} \label{HL_mixed_mat}
  \left\{
  \begin{aligned}
    \matM^{\ell-1} \arr{\sigma} - (\matD^{\ell-1}\matP^{\ell-1})^{T} \matM^{\ell} \arr{u} &= 0
    \\
    \matM^\ell \matD^{\ell-1}\matP^{\ell-1} \arr{\sigma}
    + \big(
    (\matD^\ell\matP^\ell)^T \matM^{\ell+1} (\matD^\ell\matP^\ell) 
      + \alpha_h \matS^\ell \big) \arr{u}  + (\matP^\ell)^T \matM^{\ell,\frH} \arr{p} 
        &= (\matP^\ell)^T \tilde {\arr{f}}
    \\
    (\matM^{\ell,\frH})^T \matP^\ell \arr{u} &= 0
  \end{aligned}~.
  \right.
\end{equation}
Here, $\arr{\sigma}$, $\arr{u}$ and $\arr{p}$ are the column arrays containing the coefficients of 
$\sigma_h$, $u_h$ and $p_h$ in the bases of $V^{\ell-1}_h$, $V^{\ell}_h$ and $\frH^{\ell,c}_h$ and $\tilde {\arr{f}}$ 
is the array containing the moments of $f$ against the basis functions of $V^\ell_h$,
as in \eqref{arrf_tarrf},
so that $(\matP^\ell)^T \tilde {\arr{f}}$ corresponds to the moments of the dual commuting projection $\tilde \pi^\ell_h f$.

In the case where the harmonic space $\frH^{\ell,c}_h$ is trivial, we note that the above problem
may be written in a simpler form
\begin{equation}\label{Lh_pbm_noH}
(\matM^{\ell} \matL^\ell + \alpha_h \matS^\ell) \arr{u} = (\matP^\ell)^T \tilde {\arr{f}}
\end{equation}
(completed with $\arr{\sigma} = (\matM^{\ell-1})^{-1}(\matD^{\ell-1}\matP^{\ell-1})^{T} \matM^{\ell} \arr{u}$
and $\arr{p} = 0$),
where 
\begin{equation}\label{Lh_mat}
    \matL^\ell := \matD^{\ell-1}\matP^{\ell-1} (\matM^{\ell-1})^{-1} (\matD^{\ell-1}\matP^{\ell-1})^{T} \matM^{\ell}
        + (\matM^{\ell})^{-1} (\matD^{\ell}\matP^{\ell})^T \matM^{\ell+1} \matD^{\ell}\matP^\ell
\end{equation}
is the matrix of the unpenalized Hodge Laplacian operator 
$L^\ell_{h} = d^{\ell-1}_h d^*_{\ell,h} + d^*_{\ell+1,h} d^\ell_h$. 
Equation \eqref{Lh_pbm_noH} corresponds to $L^\ell_{h,\alpha_h} u_h = \tilde \pi^\ell_h f$ 
tested against the basis functions of $V^\ell_h$, which according to our analysis 
is well-posed for $\alpha_h > 0$ in the absence of harmonic forms. 
Again, due to the cell-diagonal structure of the mass matrices, 
we observe that the problem matrix appearing the LHS of \eqref{Lh_pbm_noH} is local
in the sense of Theorem~\ref{th:loc_bfeec}.

\subsection{Numerical results}

In this section we assess the accuracy and robustness of the above discretization
for the Hodge Laplacian operator corresponding to the vector-valued case in 2D, i.e.
$L^1 = d^0 d^*_1 + d^*_2 d^1 = - \bgrad \Div + \bcurl \curl$.
Working with $\Omega = ]0,2\pi[^2$ and adopting the boundary conditions
from \eqref{derham_2D}--\eqref{derham_dual},
we find that the domain space \eqref{DL} is
\begin{equation} \label{DL1}
  D(L^1) = \{ v \in H_0(\curl;\Omega)  : \Div v \in H^1_0(\Omega), \curl v \in H(\bcurl;\Omega) \}.
\end{equation}
We first consider a Helmholtz-like source problem
\begin{equation} \label{Helmholtz_pbm}
  -\omega^2 u + L^1 u = f
\end{equation}
which corresponds to extending the plain Hodge Laplace problem \eqref{HL_f}
to the case of a sign-indefinite operator.
Here the source is smooth,
\begin{equation} \label{Helmholtz_f}
f(\bx) = \begin{pmatrix}
  - \sin(2x_2) \cos(x_1) \big((13 - \omega^2) \cos^2(x_1) - 6\big)
  \\ \noalign{\smallskip}
  \phantom{-}
    \sin(2x_1) \cos(x_2) \big((13 - \omega^2) \cos^2(x_2) - 6\big)
\end{pmatrix}
\qquad \text{ on } ~ \Omega = ]0,2\pi[^2
\end{equation}
and the parameter is taken as $\omega = 3.5$, so that $\omega^2$ is not an eigenvalue of the Hodge Laplacian operator.
As the domain is contractible there are no harmonic fields and the problem is well-posed. The exact solution is
\begin{equation*}
  u(\bx) = \begin{pmatrix}
  - \sin(2x_2) \cos^3(x_1)
  \\ \noalign{\smallskip}
  \phantom{-}
    \sin(2x_1) \cos^3(x_2)
\end{pmatrix} \in D(L^1).
\end{equation*}
In Figure~\ref{fig:HL_conv} we show the $L^2$ convergence curves corresponding to a
conforming approximation \eqref{Lh_c} of the Hodge Laplacian operator (left plot)
and a CONGA (broken FEEC) approximation \eqref{Lh} (right plot), using the polynomial elements
described in this section, with degrees $p=1$ to 4 as indicated.
We observe that the convergence rates of both methods are similar, although some reduction
in the accuracy can be seen for the low order CONGA solutions. (In the lowest order case
the convergence of the nonconforming scheme is not clear from this plot but on a finer mesh 
corresponding to $K = 160$ the accuracy improves with a rate of $1.56$, 
comparable to the rate of $1.8$ shown by the conforming solutions between the grids $K=40$ and $80$). 

Here the penalization parameter has been set to
\begin{equation} \label{strong_alpha}
\alpha_h = \frac{10 (p + 1)^2}{h}
\end{equation}
as motivated by \cite{Buffa.Perugia.Warburton.2009.jsc}.
For completeness we have also run the same problem with weaker penalization parameters
such as $\alpha_h = 1$ or even $0$
(in the latter case the broken Hodge Laplacian operator \eqref{Lh} has a large kernel but
the Helmholtz source problem \eqref{Helmholtz_pbm}--\eqref{Helmholtz_f} is still well-posed).
As expected these choices lead to larger jumps in the broken solution $u_h$, but by
measuring the conforming error $\norm{P_h u_h - u}$ we recover almost identical convergence
curves, which is a practical evidence of the robustness of the method with respect to the
penalization parameter.

\begin{figure} 
  \includegraphics[width=0.49\textwidth]{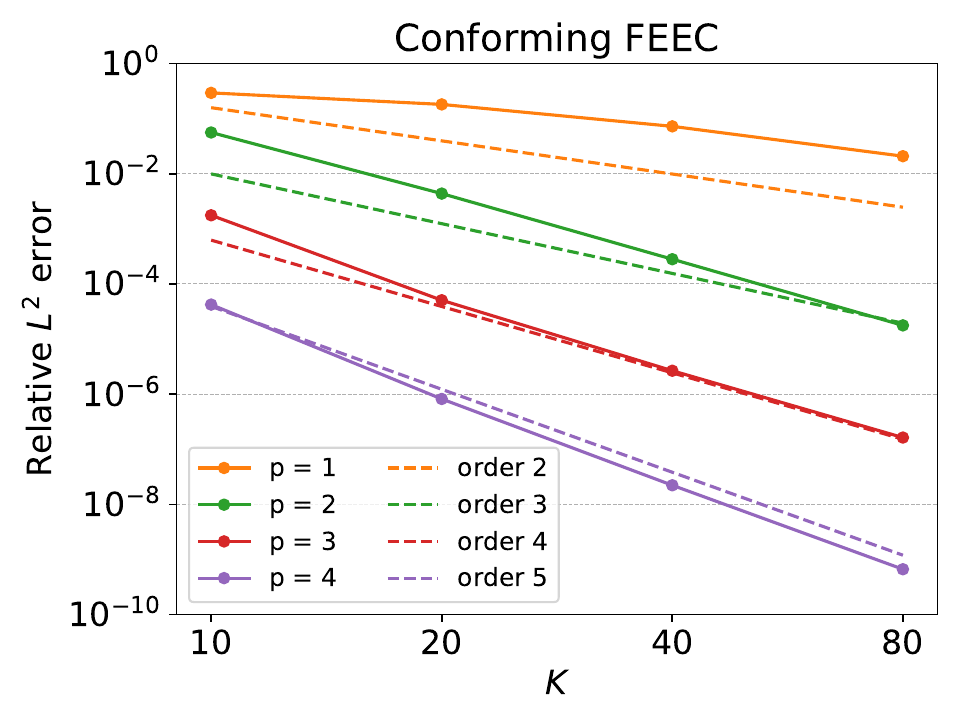}
  \hfill
  \includegraphics[width=0.49\textwidth]{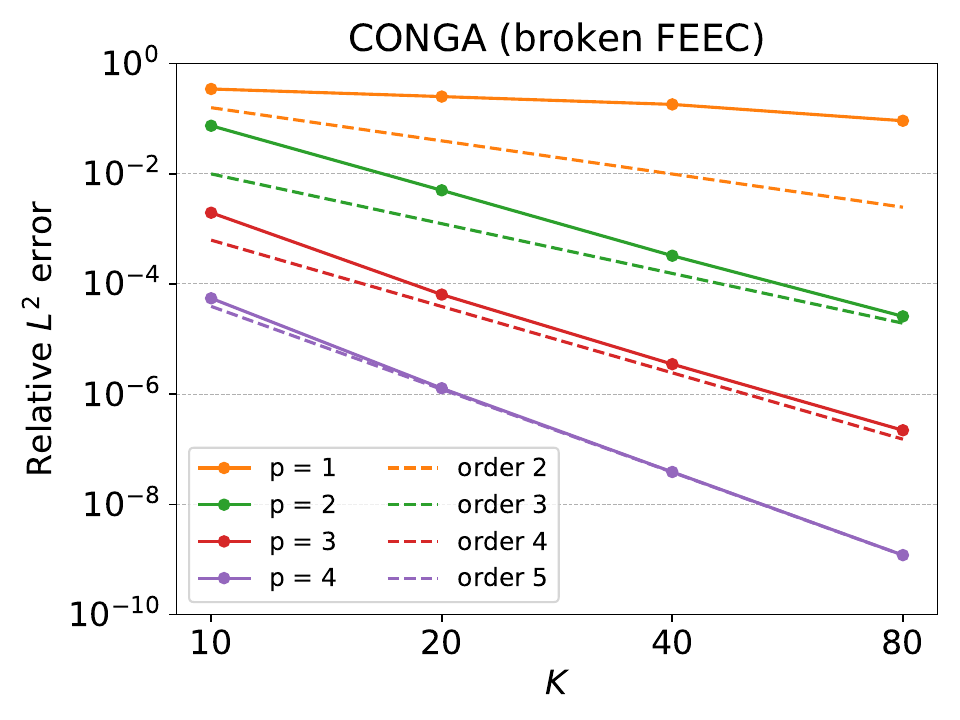}
  \caption{
  Convergence curves for the Hodge Laplace source problem, discretized with a conforming
 (left) and broken FEEC (right) method. For the latter, various penalization regimes
lead to similar curves, as described in the text.
 On the horizontal axis we report the number of cells $K$ along each direction, hence $h = 2\pi/K$; on the vertical axis we report the $L^2$ norm of the error, divided by the $L^2$ norm of the exact solution.
  }
  \label{fig:HL_conv}
\end{figure}

We next study 
the CONGA (broken FEEC) approximation of the
eigenproblem \eqref{HL_eig}.
On the square $\Omega = ]0,2\pi[^2$, the eigenvalues of the Hodge Laplacian operator
$L^1 = -\bgrad \Div + \bcurl \curl$ with $D(L^1)$ given by \eqref{DL1},
read $\lambda_\bn = \frac 14 (n_1^2+n_2^2)$ for $\bn \in \NN^2$, with separable eigenmodes of the form
\begin{equation*}
u_{1,\bn} = \begin{pmatrix} \cos\big(\frac{n_1 x_1}{2}\big) \sin\big(\frac{n_2x_2}{2}\big) \\ 0 \end{pmatrix}
\quad \text{ and } \quad
u_{2, \bn} = \begin{pmatrix} 0 \\ \sin\big(\frac{n_1 x_1}{2}\big) \cos\big(\frac{n_2x_2}{2}\big)\end{pmatrix}
\end{equation*}
where we must of course discard the zero fields, and in particular the index $\bn = (0,0)$.
Note that each eigenvalue comes with an even multiplicity, namely 2 if $n_1 = n_2$ or $n_1n_2 = 0$, or higher.
In Figure~\ref{fig:HL_eigs} we show the eigenvalues of the discrete CONGA operator \eqref{HL_eig_h}
corresponding to spaces of degree $p= 2$ and different mesh sizes. On the left panel we show
the first 40 eigenvalues of the operator associated with the penalization parameters \eqref{strong_alpha}.
The convergence there seems to be rapid, which is confirmed in Figure~\ref{fig:HL_eigserr} where
the eigenvalue errors $\abs{\lambda_\bn - \lambda_{\bn,h}}$ are plotted in logscale for the same parameters.
These results can be compared with the analog quantities corresponding to a penalization $\alpha_h = 1$,
shown on the right panels of Figure~\ref{fig:HL_eigs} and \ref{fig:HL_eigserr}.
For this weak penalization regime we see a clear issue: the first six eigenvalues do not
converge, but rather stagnate around a value close to 1.22 which is related with the jump
penalization operator $(I-P^*_h)(I-P_h)$. This spurious eigenvalue has actually a large multiplicity
(indeed the 40 eigenvalues shown here are those which are closer to the exact ones, but not
the smallest ones) and other runs have shown that it varies with the degree $p$.

Overall, these results provide a numerical validation of Theorems~\ref{th:est} and \ref{th:spectral}.
We finally show in Figure~\ref{fig:HL_eigsnope} the non-zero eigenvalues and associated 
eigenvalue errors obtained with the unpenalized operator corresponding to the limit value $\alpha_h = 0$.
As expected the unpenalized operator has a large (spurious) kernel, but by plotting 
the first 40 non-zero eigenvalues we obtain curves which are virtually on top of those 
of the strongly penalized case, seen in the left panels of Figure~\ref{fig:HL_eigs} 
and \ref{fig:HL_eigserr}. Thus, these numerical results suggest that the
nontrivial discrete spectrum converges towards the exact one with no need of a stabilization.
If true, this result would extend a similar one proven in \cite{Campos-Pinto.Sonnendrucker.2016.mcomp}
for the unpenalized CONGA approximation of the curl-curl eigenvalue problem. 
Although the unpenalized case is not covered by the analysis presented in this article, 
our explanation for this behaviour is that in the weakly penalized case the spurious eigenvalues are 
close to the correct ones, and their large multiplicity leads to a mixing of the spurious and correct eigenspaces.
In the unpenalized case the spurious eigenmodes essentially correspond to the kernel of $L_{h,0}$, hence they are 
clearly separated from the correct ones which are associated with positive eigenvalues 
$\lambda_h = (\norm{d^*_h u_h}^2+ \norm{d_h u_h}^2)/\norm{u_h}^2 > 0$.
In particular the respective eigenspaces are orthogonal, owing to the symmetry of the discrete Hodge Laplacian operator.

\begin{figure} 
    \includegraphics[width=0.47\textwidth]{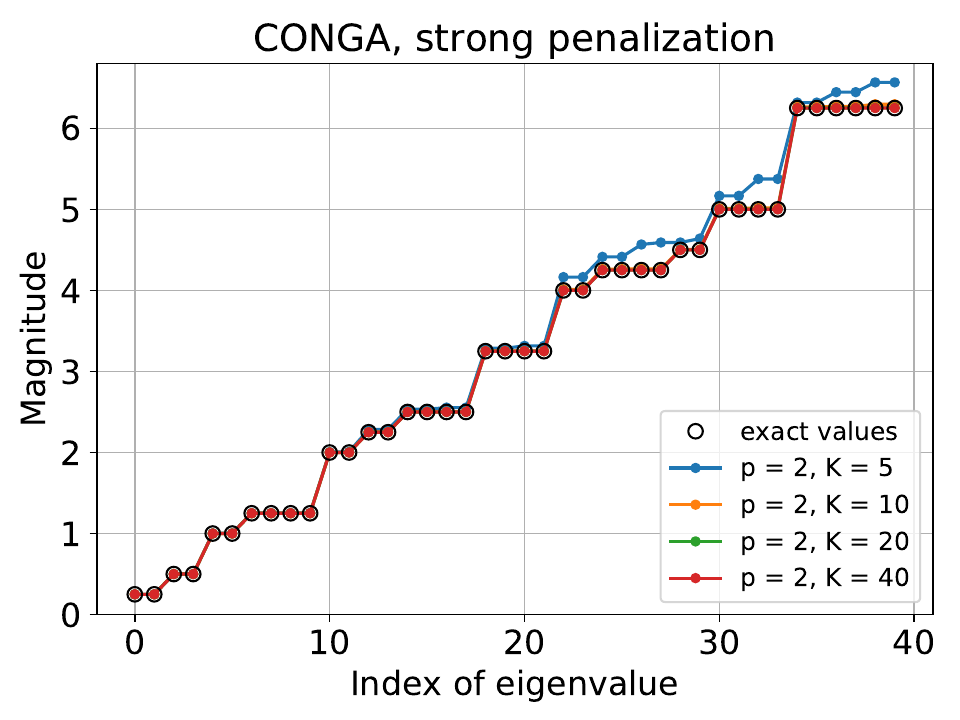}
    \hfill
    \includegraphics[width=0.47\textwidth]{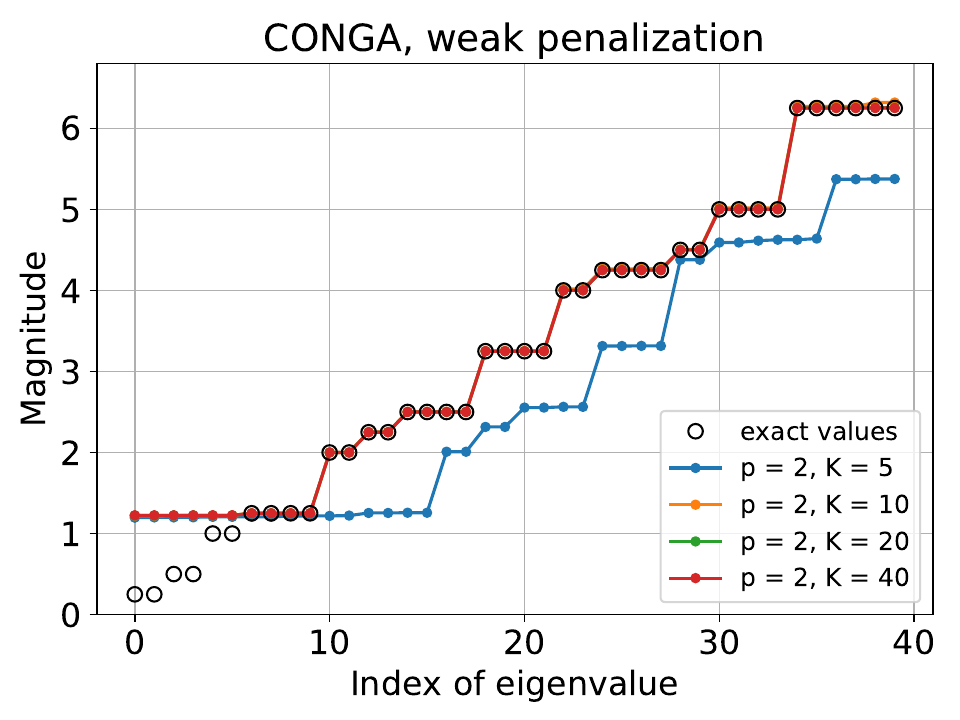}
  \caption{
  Discrete eigenvalues of the CONGA Hodge Laplacian operator on the square,
  with strong (left) and weak (right) penalization regimes, as discussed in the text.
  }
  \label{fig:HL_eigs}
\end{figure}

\begin{figure} 
    \includegraphics[width=0.49\textwidth]{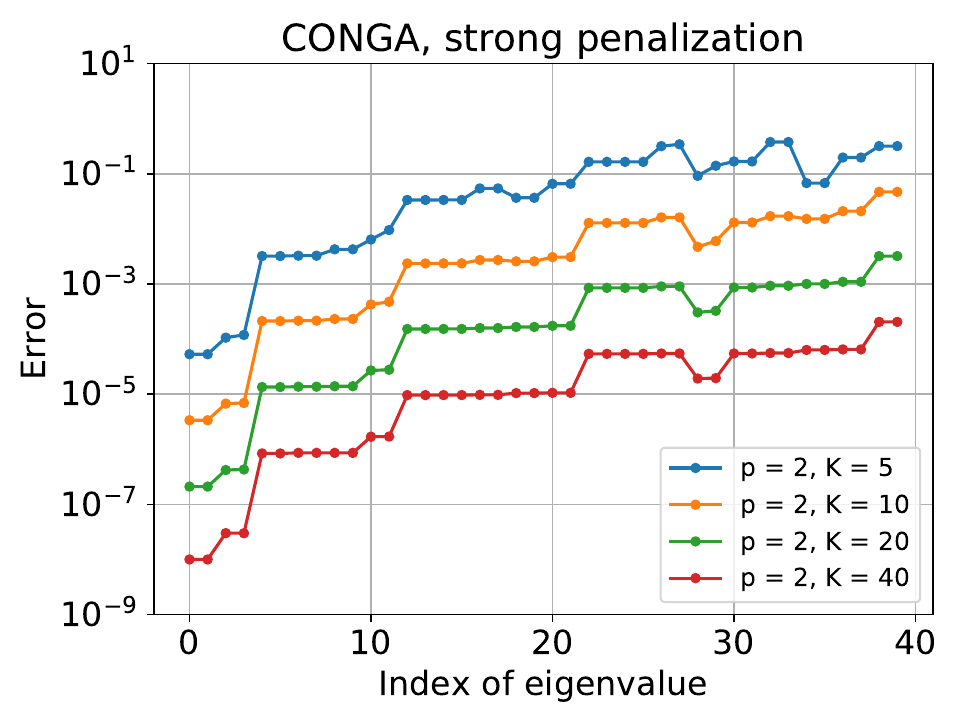}
    \hfill
    \includegraphics[width=0.49\textwidth]{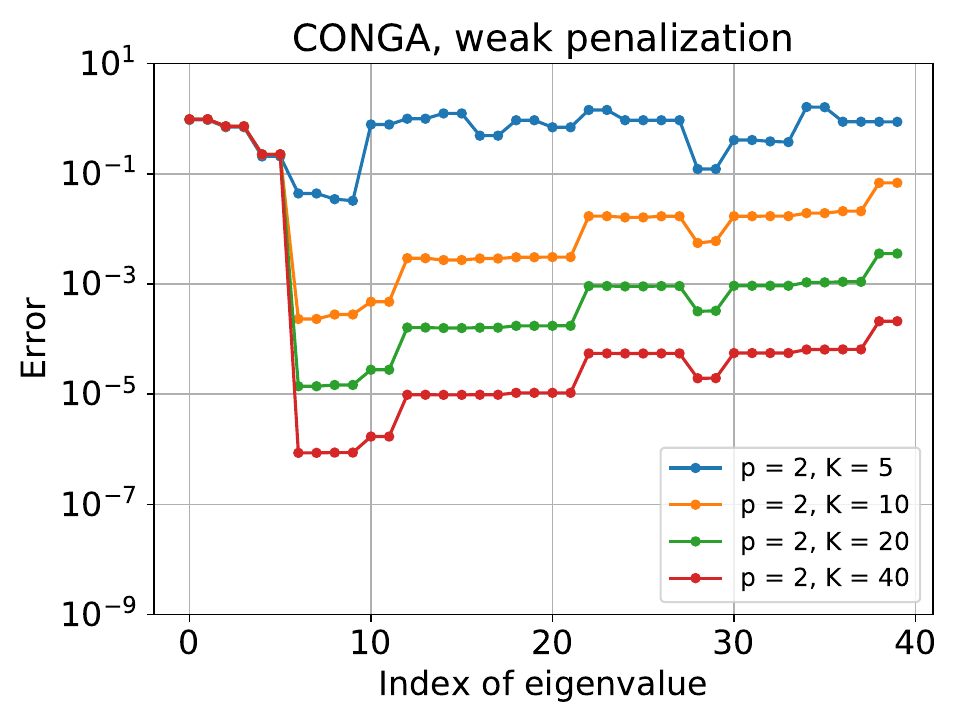}
  \caption{
  Eigenvalue errors for the CONGA Hodge Laplacian operator on the square,
  for the same parameters as in Figure~\ref{fig:HL_eigs}.
  }
  \label{fig:HL_eigserr}
\end{figure}

\begin{figure} 
    \includegraphics[width=0.49\textwidth]{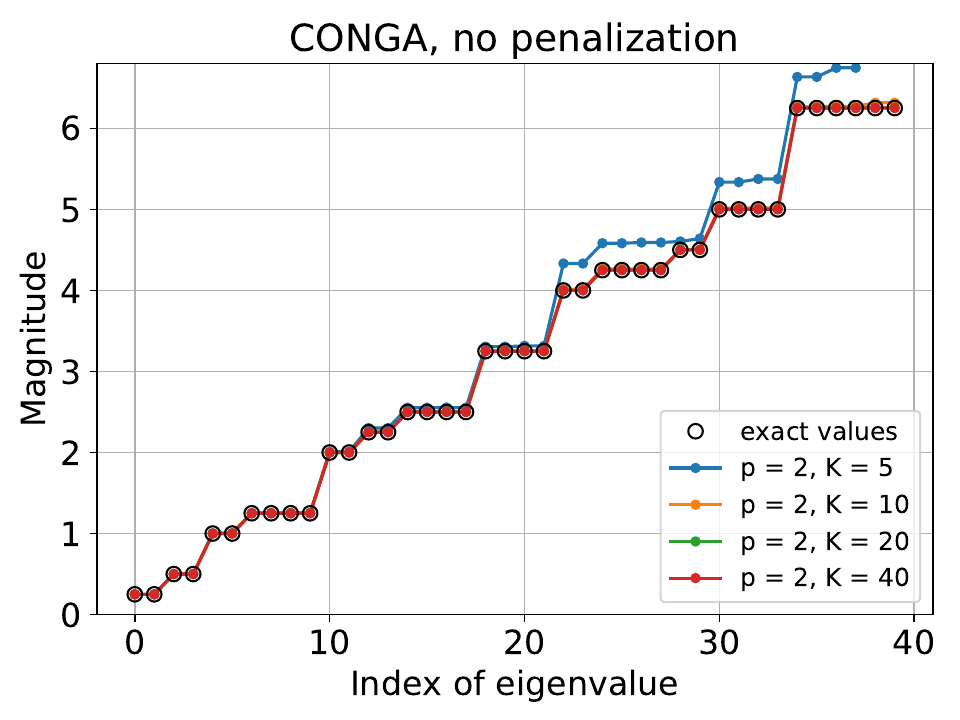}
    \hfill
    \includegraphics[width=0.49\textwidth]{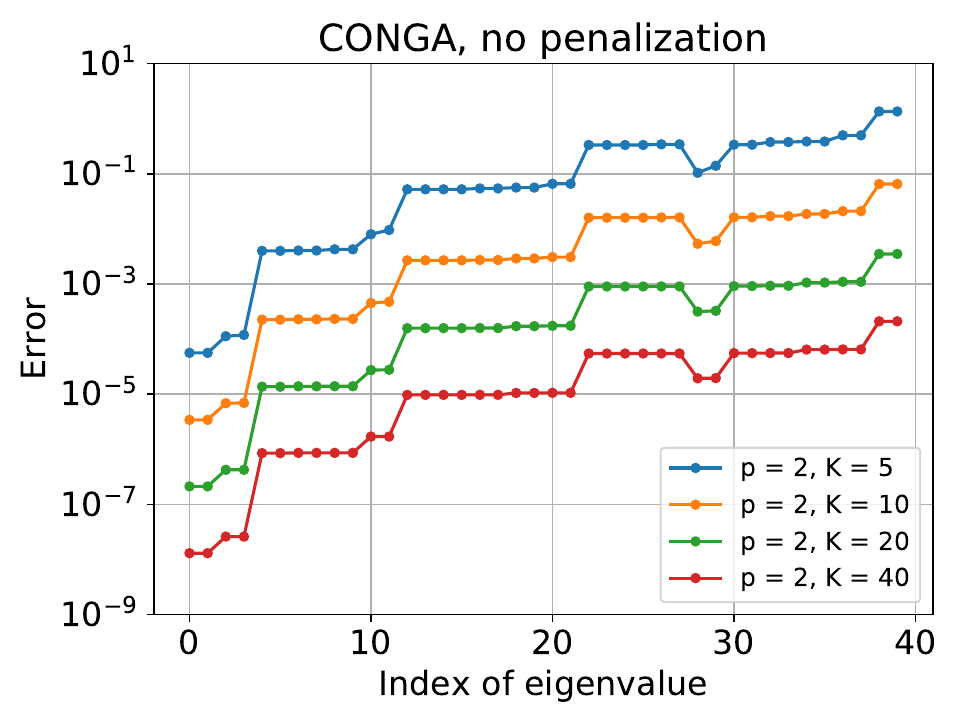}
  \caption{
  Positive eigenvalues and errors for the unpenalized ($\alpha_h = 0$)
  CONGA Hodge Laplacian operator on the square.
  }
  \label{fig:HL_eigsnope}
\end{figure}

\section{Conclusion}

In this article we have studied general discretizations of Hilbert complexes $(V, d)$
where the conformity constraint is relaxed while preserving most of the intrinsic 
stability and structure-preservation properties of {conforming} 
Finite Element Exterior Calculus (FEEC) discretizations. 
In our approach the {nonconforming} complexes $(V_h, d_h)$ 
rely on stable {conforming} discretizations,
i.e. discrete subcomplexes $(V^c_h, d)$ admitting a bounded cochain projection $\pi_h: V \to V^c_h$, 
and the discrete differential operators $d_h = dP_h$ are based on 
{\em conforming projections} $P_h: V_h \to V_h$ which are stable projection 
operators onto the underlying subcomplex $V^c_h = V_h \cap V$.

This {\em broken FEEC} approach has been originally introduced in the
conforming/nonconforming Galerkin (CONGA) approximations to 
time-dependent Maxwell equations \cite{Campos-Pinto.Sonnendrucker.2016.mcomp},
where the curl-conformity constraint was relaxed while preserving the de Rham structure
properties in the case of an exact sequence.
Here we have extended this approach to the discretization of full Hilbert complexes
with nontrivial harmonic spaces, and we have completed the construction with stable 
commuting projections $\tilde \pi_h$ for the dual (weak) discrete sequence.
We have also studied the properties of the associated CONGA Hodge Laplacian operator,
in particular we have shown that it has the same kernel as the underlying conforming FEEC 
Hodge Laplacian operator (namely, conforming discrete harmonic fields),
provided an arbitrary stabilization term is used to handle the space nonconformities.

In a second part, we have studied the CONGA Hodge Laplace source and eigenvalue problems.
Under stability and moment-preserving assumptions for the conforming projections
we have shown that the source problem is well--posed, and we have derived a priori error 
estimates that allowed us to demonstrate the spectral correctness 
of the penalized CONGA Hodge Laplacian operator. This guarantees in particular that the latter 
is free of spurious eigenvalues. 

In a third part we have applied our broken FEEC discretization method to polynomial 
finite elements where the cell-diagonal structure of the mass matrices naturally yields 
local discrete differential operators, both for the primal (strong) and dual (weak) sequences,
as well as local $L^2$ stable dual commuting projection operators.
Finally we have validated our theoretical study with numerical examples on a simple square domain,
and we have assessed their dependency with respect to the stabilization parameter.
We point out that these results have been extended to mapped spline elements 
on multi-patch non-contractible domains in a recent article \cite{conga_psydac} 
where CONGA schemes have been also proposed for several electromagnetic problems, 
including time-harmonic Maxwell and magnetostatic problems.
Our results also seem to indicate that the CONGA Hodge Laplacian is spectrally correct
in the absence of stabilization. This property however falls outside the scope of the
present analysis, and remains a subject for further investigation.

\section*{Acknowledgments}

The authors would like to thank Eric Sonnendrücker for inspiring discussions throughout this work.
They also thank the anonymous reviewer for several useful suggestions that helped
improve the presentation of our results, including a tentative explanation for the spectral correctness 
observed in the unpenalized case.
The work of Yaman Güçlü was partially supported by the European Council under the Horizon 2020
Project Energy oriented Centre of Excellence for computing applications - EoCoE, Project ID 676629.

\bibliographystyle{amsplain}
\bibliography{mc_ch}

\end{document}